\documentclass[reqno,12pt,letterpaper]{amsart}
\RequirePackage{amsmath,amssymb,amsthm,graphicx,mathrsfs,url,slashed}
\RequirePackage[usenames,dvipsnames]{xcolor}
\RequirePackage[colorlinks=true,linkcolor=Red,citecolor=Green]{hyperref}
\RequirePackage{amsxtra}
\usepackage{cancel}
\usepackage{tikz-cd}

\title{Notes on Harmonic maps}
\author{Georgios Daskalopoulos and Chikako Mese}
\date{June 2022}
\thanks{
GD supported in part by NSF DMS-2105226, CM supported in part by NSF DMS-2005406. 
We would like to thank Yitong Sun for carefully reading  this document and making useful suggestions to improve the exposition of this paper.}

\newcommand{\mres}{\mathbin{\vrule height 1.6ex depth 0pt width
0.13ex\vrule height 0.13ex depth 0pt width 1.3ex}}

\newcommand{\RR}{\mathbb{R}}
\newcommand{\CC}{\mathbb{C}}

\newcommand{\Sph}{\mathbf S}
\newcommand{\calM}{\mathcal M}

\newcommand{\mfld}{\mathcal{H}}

\newcommand{\Ric}{\mathrm{Ric}}

\newcommand{\dfn}[1]{\emph{#1}\index{#1}}

\newcommand{\loc}{\mathrm{loc}}

\newcommand{\parl}{\left(}
\newcommand{\parr}{\right)}
\newcommand{\bral}{\left\langle}
\newcommand{\brar}{\right\rangle}

\newtheorem{theorem}{Theorem}[section]

\newtheorem{lemma}[theorem]{Lemma}

\newtheorem{proposition}[theorem]{Proposition}
\newtheorem{corollary}[theorem]{Corollary}

\theoremstyle{definition}
\newtheorem{definition}[theorem]{Definition}
\newtheorem{remark}[theorem]{Remark}
\newtheorem{example}[theorem]{Example}
\newtheorem{examples}[theorem]{Examples}



\numberwithin{equation}{section}

\def\XXint#1#2#3{{\setbox0=\hbox{$#1{#2#3}{\int}$ }
\vcenter{\hbox{$#2#3$ }}\kern-.6\wd0}}

\begin{document}

\maketitle



This is  set of notes prepared for the Summer School on non-Abelian Hodge theory in Abbaye de Saint-Jacut de la Mer 
June, 6-19,  2022.

\vspace*{1in}

\begin{center}
{\sc Table of Contents}
\end{center}

\begin{itemize}
\item[] Lecture~\ref{sec:smooth harmonic maps}:  {\sc Harmonic Maps Between Riemannian Manifolds}   \hfill p.~\pageref{sec:smooth harmonic maps} 
\item[] Lecture~\ref{sec:NPC}: {\sc {Existence and regularity}}
 \hfill p.~\pageref{sec:Existence Results}
 \item[] Lecture~\ref{sec:siu}: {\sc Pluriharmonic Maps and the Siu-Sampson Formula} \hfill  p.~\pageref{sec:siu}
\item[] Lecture~\ref{sec:Donaldson}:  {\sc Donaldson Corlette Theorem} \hfill p.~\pageref{sec:Donaldson}
\end{itemize}

\newpage

\section{Harmonic Maps Between Riemannian manifolds}
\label{sec:smooth harmonic maps}

\subsection{Introduction: Basics}
In this section, we define energy of maps between Riemannian manifolds, harmonic maps, and the first and second variation formulas after the pioneering work of 
Eells-Sampson \cite{eells-sampson}.  A good reference is also  \cite{jost}.
 
\subsection{The energy of maps}
Let  $(M,g)$, $(N,h)$ be Riemannian manifolds.
Let  $f: M \to N$ be a smooth map which induces a map $df:TM \rightarrow TN$
$$df \parl \frac{\partial}{\partial x^\alpha} \parr \bigg|_p = \frac{\partial f^i}{\partial x^\alpha} \frac{\partial}{\partial y^i} \bigg|_{f(p)}
$$
where $(x^\alpha)$ (resp.~$(y^i)$) are the local coordinates of $M$ (resp.~$N$).

The map $f$ also induces a vector bundle $f^*TN$ over $M$. Let $\nabla$ be a connection on $f^*TN$ inherited from the Levi-Civita connection on $TN$.
Then $\nabla$ induces an exterior derivative
\[
d_\nabla:  C^{\infty}((\Lambda^p T^*M) \otimes  f^*TN) \rightarrow C^{\infty}((\Lambda^{p+1} T^*M) \otimes  f^*TN).
\]

We view $d f$ as a section 
$$
d f \in 
C^\infty(T^* M \otimes f^* TN)= \Omega^1(f^* TN).
$$
Using the notation
$$
\frac{\partial}{\partial f^i} := \frac{\partial}{\partial y^i} \circ f \in C^\infty_\loc(M, f^* TN),
$$
we have
$$d f = \frac{\partial f^i}{\partial x^\alpha} d x^\alpha \otimes \frac{\partial}{\partial f^i}.$$

Let  $(g_{\alpha\beta})$  (resp.~$(h_{ij})$) be the expression of the Riemannian metric $g$  of $M$ (resp.~ $h$ of $N$) with respect to local coordinates $(x^\alpha)$  (resp.~$(y^i)$).
\begin{definition}
Set 
$$e(f) := \frac{1}{2} |d f|^2 = \frac{1}{2} \frac{\partial f^i}{\partial x^\alpha} \frac{\partial f^j}{\partial x^\beta} g^{\alpha \beta} h_{ij} \circ f.$$
The \dfn{energy} of $f$ is 
$$E(f) := \int_M e(f) \star 1 = \frac{1}{2} \int_M g^{\alpha \beta}(x) h_{ij}(f(x)) \frac{\partial f^i}{\partial x^\alpha} \frac{\partial f^j}{\partial x^\beta} \sqrt{g(x)} \ d x^1 \wedge \cdots \wedge d x^n.$$
Here, recall that the {\it Hodge star operator} $\star:\Lambda^kTM\to \Lambda^{n-k}TM$, is the unique linear operator such that for all $\alpha,\beta\in \Lambda^{k}V$,
\[\alpha\wedge\star\beta = g(\alpha,\beta)\star1.\] 
\end{definition}

\begin{lemma} \label{lem:varvf}
Let $f=(f_t)$ be a  smooth one-parameter family of $C^\infty$ maps 
\[
f:M \times (-\epsilon,\epsilon) \rightarrow N, \ \ \  f(x,t)=f_t(x).
\]
Then
$$\nabla \frac{\partial f}{\partial t} = \nabla_{\partial/\partial t} d f$$
where $f = f_t$ and 
$$\frac{\partial f}{\partial t} = \frac{\partial f^i}{\partial t} \frac{\partial}{\partial f^i} \in C^\infty(f^*TN).$$
\end{lemma}

\begin{proof}
Both $\nabla \frac{\partial f}{\partial t}=\nabla_{{\partial}/{\partial x^\alpha}} \frac{\partial f}{\partial t} dx^\alpha$ and $\nabla_{\partial/\partial t} d f=\nabla_{\partial/\partial t} \frac{\partial f}{\partial x^\alpha} dx^\alpha $ are 1-forms with values in $f^*TN$.  Here, 
\[
\frac{\partial f}{\partial x^\alpha} = \frac{\partial f^j}{\partial x^\alpha} \frac{\partial}{\partial f^j} \in C^{\infty}(f^*TN).
\]
 Consider $f$ as a map $f:M \times (-\epsilon,\epsilon) \rightarrow N$.  Since
\[
 \left[f_*\left(\frac{\partial}{\partial t}\right), f_*\left( \frac{\partial}{\partial x^\alpha} \right) \right]  =f_*\left[\frac{\partial}{\partial t}, \frac{\partial}{\partial x^\alpha} \right] =0
\]
and $\nabla$ is torsion-free, 
\[
\nabla_{{\partial}/{\partial x^\alpha}} \frac{\partial f}{\partial t}= \left( \nabla_{f_*\left(\frac{\partial}{\partial x^\alpha}\right)} f_*\left(\frac{\partial}{\partial t} \right) \right)\circ f
=\left(\nabla_{f_*(\frac{\partial f}{\partial t})} f_*\left(\frac{\partial }{\partial x^\alpha} \right)  \right)\circ f 
=\nabla_{\partial/\partial t} d f(\frac{\partial}{\partial x^\alpha} )
\]
which proves the equality.

\end{proof}

\begin{corollary}[First Variation Formula]
For $(f_t)$ as above, 
$$\frac{d}{d t} E(f_t) = \int_M \bral \nabla \frac{\partial f}{\partial t}, d f\brar \star 1. $$
\end{corollary}
\begin{proof}
We compute 
\begin{align*}
\frac{d}{d t} E(f_t) &= \frac{1}{2} \int_M \frac{d}{d t} \bral d f, d f \brar \star 1\\
&= \int_M \bral \nabla_{\partial/\partial t} d f, d f \brar \star 1\\
&= \int_M \bral \nabla \frac{\partial f}{\partial t}, d f \brar \star 1. \qedhere
\end{align*}
\end{proof}

\begin{corollary}
The critical points $f$ of the functional  $E$ satisfy
\begin{equation}\label{eqn:firstvargen}
\int_M\bral  \nabla \psi, df  \brar \star1=0, \ \ \ \forall \psi \in C^\infty(f^* TN).
\end{equation}
By taking $\psi$  compactly supported away from $\partial M$, we obtain the Euler Lagrange equation of $E$, 
\begin{equation}\label{eqn:firstvar3}
\tau(f):=-d_\nabla^{\star} df= \star d_\nabla \star df =0.
\end{equation}
Here, $\nabla$ is the pullback of the Levi-Civita connection on $f^*(TN)$. 
\end{corollary}
\begin{proof}
Let $f_t$ be a family of maps with 
$$\frac{d}{d t}\bigg|_{t = 0} f_t = \psi \in C^\infty(f^* TN).$$
Taking $\psi$ compactly supported and integrating by parts,
$$\frac{d}{d t}\bigg|_{t = 0} E(f_t) = \int_M \bral \psi, -\tau f \brar \star 1 = 0$$
which holds for every $\psi \in C_c^\infty(f^* TM)$ iff $\tau f = 0$. 
\end{proof}

\begin{definition}
A smooth map $f:M \rightarrow N$ satisfying $d_\nabla^\star df=0$ is  called a \emph{harmonic map}.
\end{definition}

\subsection{Harmonic map equations in local coordinates}
First define
$$\omega\left(\frac{\partial}{\partial y^i}\right) := \Gamma_{jk}^i d y^k \otimes \frac{\partial}{\partial y^j},$$
where $\Gamma_{jk}^i$ are Christoffel symbols on $N$ and set 
$$\tilde \omega := \omega \circ f = \left( \Gamma_{ij}^k d y^k \otimes  \frac{\partial}{\partial y^j} \right)\circ f
=( \Gamma_{ik}^j \circ f ) \frac{\partial f^k}{\partial x^\beta} d x^\beta \otimes \frac{\partial}{\partial f^j} 
.$$
Then  $d_\nabla=d+\tilde \omega$ and 
\begin{align*}
d_\nabla^\star d f &= -\star d_\nabla \star  d f \\
&= -\star  (d + \tilde \omega)\parl\star  d f^i \otimes \frac{\partial}{\partial f^i}\parr \\
&= -( \star  d \star  d f^i ) \frac{\partial}{\partial f^i} - (-1)^{m - 1} \star \parl\star  d f^i \otimes \tilde \omega\parl\frac{\partial}{\partial f^i}\parr\parr\\
&= -\Delta f^i  \ \frac{\partial}{\partial f^i} - (-1)^{m - 1} \star  \parl \frac{\partial f^i}{\partial x^\alpha} \star  d x^\alpha \wedge ( \Gamma_{ik}^j \circ f ) \frac{\partial f^k}{\partial x^\beta} d x^\beta \otimes \frac{\partial}{\partial f^j} \parr
\\
&= -\Delta f^i  \ \frac{\partial}{\partial f^i} - (-1)^{m - 1} \parl \frac{\partial f^i}{\partial x^\alpha} \frac{\partial f^k}{\partial x^\beta} \Gamma_{ik}^j \circ f  \star  (\star  d x^\alpha \wedge d x^\beta)  \otimes \frac{\partial}{\partial f^j} \parr \\
&= -\left( \Delta f^k   + g^{\alpha \beta} \frac{\partial f^i}{\partial x^\alpha} \frac{\partial f^j}{\partial x^\beta} \Gamma_{ij}^k \circ f  \right)\frac{\partial}{\partial f^k}.
\end{align*}
Thus the harmonic map equation is 
\begin{equation} \label{hme}
\Delta f^k + g^{\alpha \beta} \frac{\partial f^i}{\partial x^\alpha} \frac{\partial f^j}{\partial x^\beta} \Gamma_{ij}^k \circ f = 0.
\end{equation}

\begin{examples}
\hspace*{0.05in}

\noindent (1) Suppose $N=\RR$, then the harmonic map equation reduces to $\Delta f=0$; i.e.~$f$ is a harmonic function on $M$.
 
\noindent (2)
Suppose $M = \Sph^1$.
Then 
$$E(f) = \frac{1}{2} \int_0^{2\pi} |\dot f(t)| d t$$
and the critical points of $E(f)$ are geodesics.
We can also see this from the harmonic maps equation.
Since $\Sph^1$ is $1$-dimensional we can take $g^{\alpha \beta} = \delta^{\alpha \beta}$.
Then 
$$\frac{\partial^2 f^k}{\partial t^2} + \Gamma_{ij}^k \frac{\partial f^i}{\partial x^\alpha} \frac{\partial f^j}{\partial x^\beta} = 0.$$
This is the \emph{geodesic equation}.

\noindent (3) We'll show later that holomorphic maps between K\"ahler manifolds are harmonic (cf.~Remark~\ref{holimpleishar'}).
\end{examples}

\subsection{The Dirichlet and Neumann problems} If $M$ has non-empty boundary ($N$ is without boundary) there are two different boundary value problems to consider: 
\begin{itemize}
\item {\bf Dirichlet problem:} Minimize $E$ in a fixed homotopy class of maps from $M$ to $N$ relative to the boundary of $M$. This is equivalent to considering compactly supported variations $\psi$.
\item {\bf Neumann problem:} Minimize $E$ in a fixed free homotopy class of maps from $M$ to $N$, in other words no restriction on the type of variations.
\end{itemize}

\subsection{The second variation}
The Riemannian tensor 
$$R^N: TN \times TN \times TN \to TN$$
induces and operator
$$R^N: f^* TN \times f^* TN \times f^* TN \to f^* TN$$
in the natural way.

\begin{lemma}
Let $f_t: M \to N$, and let $V$ be a vector field along $f_t$. Then 
$$\nabla_{\partial/\partial t} \nabla V = \nabla \nabla_{\partial/\partial t} V - R^N \parl d f, \frac{\partial f}{\partial t} \parr V.$$
\end{lemma}

\begin{proof}
Both 
\begin{eqnarray*}
\nabla_{\partial/\partial t} \nabla V - \nabla \nabla_{\partial/\partial t} V 
& = &
\left(  \nabla_{\partial/\partial t} \nabla_{\partial/\partial x^\alpha} V - \nabla_{\partial/\partial x^\alpha} \nabla_{\partial/\partial t}  V \right) dx^\alpha
\\
& = & \left( \left( \nabla_{f_*(\partial/\partial t)} \nabla_{f_*(\partial/\partial x^\alpha)} V - \nabla_{f_*(\partial/\partial x^\alpha)} \nabla_{f_*(\partial/\partial t)}\right) f_*  V\right) \circ f  \ dx^\alpha
\end{eqnarray*}
 and 
\[
R^N \parl d f, \frac{\partial f}{\partial t} \parr V=\left(R^N \parl f_*\left(\frac{\partial}{\partial x^\alpha}\right), f_*\left( \frac{\partial }{\partial t} \right) \parr f_*(V) \right) \circ f
\ dx^\alpha
\]
 are $1-$forms on $M$ with values in $f^*TN$.  Thus, assertion follows from the definition of the Riemannian tensor $R^N$.

\end{proof}

\begin{theorem}[Second Variation Formula] \label{secvar}
One has 
$$\frac{d^2}{d t^2} E(f_t) = ||\nabla \frac{\partial f}{\partial t}||^2 - \int_M \bral R^N \parl d f, \frac{\partial f}{\partial t} \parr \frac{\partial f}{\partial t}, d f \brar \star 1 + \int_M \bral \nabla_{\partial/\partial t} \frac{\partial f}{\partial t}, \tau f \brar \star 1.$$
\end{theorem}

\begin{proof}
We compute 
\begin{align*}
\frac{d^2}{d t^2} E(f_t) &= \int_M \frac{d}{d t} \bral \nabla \frac{\partial f}{\partial t}, d f \brar \star 1 \\
&= \int_M \parl \bral \nabla_{\partial/\partial t} \nabla \frac{\partial f}{\partial t}, d f \brar + \bral \nabla \frac{\partial f}{\partial t}, \nabla_{\partial/\partial t} d f \brar \parr \star 1 \\
&= \int_M \parl \bral \nabla \nabla_{\partial/\partial t} \frac{\partial f}{\partial t}, d f \brar - \bral R^N \parl d f, \frac{\partial f}{\partial t} \parr \frac{\partial f}{\partial t}, d f \brar + ||\nabla \frac{\partial f}{\partial t}||^2 \parr \star 1. \qedhere
\end{align*}
\end{proof}

\section{Existence and regularity}
\label{sec:NPC}

\subsection{Introduction:  Non-positive curvature}
In this section, we examine the role of non-positive curvature  of the target metric on harmonic maps.  We show  uniqueness and discuss regularity.  We also study the equivariant problem and prove existence of equivariant harmonic maps into non-positively curved metric spaces.  Some references are \cite{schoen}, \cite{korevaar-schoen1}, \cite{korevaar-schoen2} and \cite{gromov-schoen}.

\subsection{Second variation formula and non-positive curvature}

The following are corollaries of Theorem~\ref{secvar}.

\begin{corollary}
If $N$ has $\leq 0$ sectional curvature and $f_t$ is a geodesic interpolation, then $E(f_t)$ is convex.
\end{corollary}
\begin{proof}
In the second variation  formula the last term vanishes and the others are $\geq 0$.
\end{proof}

\begin{corollary}
Let $f, \phi: M \to N$ be homotopic with $f|\partial M = \phi|\partial M$.
If $N$ has $\leq 0$ sectional curvature and $f$ is harmonic, then
$$E(f) \leq E(\phi).$$
\end{corollary}
\begin{proof}
Let $f_t$ be a geodesic homotopy between $f, \phi$, thus $f_0 = f$, $f_1 = \phi$. Then $E(t) = E(f_t)$ is convex, and $E'(0) = 0$. So $E(1) \geq E(0)$, hence $E(\phi) \geq E(f)$.
\end{proof}

\begin{corollary}
If $f_0, f_1:M \rightarrow N$ are homotopic harmonic maps with $f_0|\partial M = f_1|\partial M$ and $N$ has   $\leq 0$ sectional curvature, then:
\begin{enumerate}
\item If $\partial M$ is nonempty, then $f_0 = f_1$.
\item If $\partial M$ is empty, $F$ is a geodesic homotopy  between $f_0, f_1$ and $N$ has  sectional curvature $< 0$ at one point $p$ in the image of $F$, then either $f_0 = f_1$ or the rank of $f_0$ is $\leq 1$.
\end{enumerate}
\end{corollary}
\begin{proof}
(1) Let $f_t$ be a geodesic homotopy between $f_0, f_1$, $E(t) = E(f_t)$. Then $E$ is convex.  Since $E'(0) = E'(1) = 0$, we conclude  $E' (t)= 0=E''(t)$. By Theorem~\ref{secvar},
$$\nabla \frac{\partial F}{\partial t} = 0$$
and 
$$\bral R^N \parl d f, \frac{\partial f}{\partial t} \parr \frac{\partial f}{\partial t}, d f \brar = 0.$$
Thus, 
$$\frac{\partial}{\partial x^\alpha} ||\frac{\partial F}{\partial t}||^2 = 2 \bral \nabla_{\partial/\partial x^\alpha} \frac{\partial F}{\partial t}, \frac{\partial F}{\partial t} \brar = 0$$
which implies that $||\partial F/\partial t||$ is constant.
But $\partial F/\partial t = 0$ on $\partial M$, so $\partial F/\partial t = 0$ everywhere if $\partial M$ is nonempty and hence $f_0 = f_1$.

(2) If $||\partial F/\partial t|| = 0$, then $f_0 = f_1$.
Otherwise, $\partial F/\partial t \neq 0$ for every $x, t$. The negative sectional curvature at $p$ implies  $d f$ is parallel to $\partial F/\partial t$ at $p$ and therefore everywhere.  Thus,  the image of $d f$ has dimension $\leq 1$.
\end{proof}

\subsection{The Weitzenb\"ock formula}
\begin{theorem}
Let $f: M \to N$ be a harmonic map and $(e_\alpha)$ an orthonormal frame for $TM$. Then 
\begin{align*}
\Delta e(f) &= |\nabla d f|^2 + \frac{1}{2} \bral d f (\Ric^M(e_\alpha)), d f (e_\alpha) \brar \\
&\qquad - \frac{1}{2} \bral R^N(df (e_\alpha), d f (e_\beta)) d f (e_\beta), d f (e_\alpha) \brar.
\end{align*}
\end{theorem}
\begin{proof}
Expanding out the Laplacian with respect to local coordinates, the harmonic map equation (\ref{hme}) is
$$g^{\alpha \beta} f^i_{/\alpha\beta} - g^{\alpha\beta} \, {}^M \Gamma_{\alpha \beta}^\eta f^i_{/\eta} + g^{\alpha\beta} \, {}^N \Gamma_{k\ell}^i \circ f f^k_{/\alpha} f^\ell_{/\beta}=0.$$
We use normal coordinates at  $x \in M$ and $f(x) = y$.   Thus, the metric tensors $(g_{\alpha \beta})$ and $(h_{ij})$  are Euclidean up to first order at $x$ and $y$ respectively.  Differentiating,
\begin{align*}
f^i_{/\alpha \alpha \varepsilon} &= {}^M \Gamma_{\alpha\alpha/\varepsilon}^\eta f^i_{/\eta} - {}^N \Gamma_{k\ell/m}^i f^m_{/\varepsilon} f^k_{/\alpha} f^\ell_{/\alpha} \\
&= \frac{1}{2} (g_{\alpha \eta/\alpha \varepsilon} + g_{\alpha \eta/\alpha \varepsilon} - g_{\alpha \alpha/\eta \varepsilon}) f^i_{/\eta} \\
&\qquad - \frac{1}{2} (h_{ki/\ell m} + h_{\ell i/km} - h_{k\ell/im}) f^m_{/\varepsilon} f^k_{/\alpha} f^\ell_{/\alpha}.
\end{align*}
Furthermore,
\[
g^{\alpha \beta}_{/\epsilon \epsilon}=-g_{\alpha \beta/\epsilon \epsilon}
\]
and
\[
\triangle h_{ij}(f(x))=h_{ij/kl} f^k_{/\epsilon} f^k_{/\epsilon}.
\]
Thus,
\begin{align*}
\Delta \parl \frac{1}{2} g^{\alpha \beta} h_{ij} \circ f f^i_{/\alpha} f^j_{/\beta} \parr &= \frac{1}{\sqrt g} \frac{\partial}{\partial x^\sigma} \parl \sqrt g g^{\sigma \tau} \frac{\partial}{\partial x^\tau} \parl \frac{1}{2} g^{\alpha \beta} h_{ij} \circ f f^i_{/\alpha} f^j_{/\beta} \parr \parr \\
&= f^i_{/\alpha \sigma} f^i_{/\alpha \sigma} - \frac{1}{2} (g_{\alpha \beta/\sigma \sigma} + g_{\sigma\sigma/\alpha \beta} - g_{\sigma\alpha/\sigma \beta} - g_{\sigma \alpha/\sigma \beta}) f^i_{/\alpha} f^i_{/\beta} \\
&\qquad + \frac{1}{2} (h_{ij/k\ell} + h_{k\ell/ji} - h_{ik/j\ell} - h_{j\ell/ik}) f^i_{/\alpha} f^j_{/\alpha} f^k_{/\sigma} f^\ell_{/\sigma} \\
&= f^i_{/\alpha \sigma} f^i_{/\alpha \sigma} + \frac{1}{2} \Ric^M_{\alpha \beta} f^i_{/\alpha} f^j_{/\beta} - \frac{1}{2} R^N_{ikj\ell} f^i_{/\alpha} f^j_{/\alpha} f^k_{/\sigma} f^\ell_{/\sigma}. \qedhere
\end{align*}
Here, 
\[
\Ric^M_{\alpha \beta}=g^{\delta \epsilon} R^M_{\alpha \delta \beta \epsilon}
\]
 is the Ricci tensor.
\end{proof}


\subsection{Regularity}

Assume $N$ has $\leq 0$ sectional curvature. Then 
from the Weitzenb\"ock formula, 
\begin{equation} \label{schauder estimate prep}
\Delta e(f)  \geq  -Ce(f)
\end{equation} 
where $C$ depends only on the geometry of $M$.
\begin{theorem}
If $f: M \to N$ is harmonic, and $N$ has $\leq 0$ sectional curvature, then 
$$|f|_{C^{2 + \alpha}_{loc}} \leq c$$
where $c > 0$ depends on $E(f)$ and the geometries of $M, N$.
\end{theorem}
\begin{proof}
By (\ref{schauder estimate prep}) and  Moser iteration, 
\begin{equation}\label{moser}
\sup_{B_r(p)} e(f) \leq c \int_{B_{2r}(p)} e(f) \star1= E(f)
\end{equation}
where $c$ only depends on the geometry and $r$.
Now the right-hand side of (\ref{schauder estimate prep}) is $C^0$-bounded. So by elliptic regularity, $f_i$ is $C^{1 + \alpha}$-bounded.
But then the right-hand side of (\ref{schauder estimate prep}) is $C^\alpha$-bounded, so $f_i$ is $C^{2 + \alpha}$-bounded.
\end{proof}

\begin{corollary}
If $f: M \to N$ is harmonic, and $N$ has $\leq 0$ sectional curvature, then $f \in C^\infty(M, N)$.
\end{corollary}
\begin{proof}
Keep bootstrapping with (\ref{schauder estimate prep}).
\end{proof}

\begin{theorem}
If $f:M \rightarrow N$ is a harmonic map, $M$ is compact with   Ricci curvature $\geq 0$  and $N$ has sectional curvature $\leq 0$, then $f$ is totally geodesic.
\end{theorem}

\begin{proof}
Since
\begin{align*}
0 = \int_M \triangle e(f) \star1
& = \int_M \left( |\nabla df|^2 +  \frac{1}{2} \bral d f (\Ric^M(e_\alpha)), d f (e_\alpha) \brar  \right. \\
&\qquad  \left. - \frac{1}{2} \bral R^N(df (e_\alpha), d f (e_\beta)) d f (e_\beta), d f (e_\alpha) \brar \right) \star 1,
\end{align*}
and each of the terms on the right hand side is non-negative, we have
$$ \nabla df=0.$$
\end{proof}

\label{sec:Existence Results}
\subsection{Non-positive curvature in a metric space}
A complete metric space $(X,d)$ is called an NPC space
  if the following conditions are satisfied:\\
\\
(i)  The space $(X,d)$ is a length (or geodesic) space.  That is, for any two
points $P$ and $Q$ in $X$, there exists a rectifiable curve $c$ so
that the length of $c$ is equal to $d(P,Q)$ (which we will
sometimes denote by $d_{PQ}$ for simplicity).  We call such
distance realizing curves geodesics.\\
\\
(ii)   For any geodesic triangle with vertices $P,R,Q \in X$, let $c:[0,l]
\rightarrow X$ be the arclength parameterized geodesic from $Q$ to
$R$ and let $Q_{t}=c(tl)$.  Then
\begin{equation} \label{tricomp}
d_{PQ_t}^2 \leq (1-t) d^2_{PQ}+td^2_{PR}-t(1-t)d_{QR}^2.
\end{equation}
\\
(iii) Condition (ii) implies the quadralateral comparison inequalities (cf. \cite[Corollary 2.1.3]{korevaar-schoen1})
\begin{align}
d^2_{P_tQ_t} & \leq  (1-t)d^2_{PQ}+td^2_{RS}-t(1-t)(d_{SP}-d_{QR})^2 \label{menelaus}
\\
 \label{agamemnon}
d_{Q_tP}^2+d^2_{Q_{1-t}S} &\leq d^2_{PQ}+d^2_{RS}-td^2_{QR}-2td_{SP}d_{QR}+2t^2d^2_{QR}
\end{align}

\begin{example}The main examples we will be considering are Riemannian manifolds of non-positive curvature and (locally compact) Euclidean  buildings.
\end{example}

\begin{example} \label{l2maps}
Let $(X,d)$ be an NPC space, $P \in X$ and $M$ be a compact Riemannian manifold.  Let  $Y=L^2(M,X)$ be a set of maps $f:M \rightarrow X$ such that 
\[
\int_M d^2(f,P) \star 1 <\infty.
\]
Define a distance function $d_Y$ on $Y$ by setting
\[
d_Y^2(f_0,f_1)=\int_M d^2(f_0(x), f_1(x)) \star 1.
\]
Then $(Y,d_Y)$ is an NPC space (cf.~\cite[ Lemma 2.1.2]{korevaar-schoen1}) where the geodesic between $f_0$ and $f_1$ is the geodesic interpolation map $f_t(x)=(1-t)f_0(x)+tf_1(x)$.
\end{example}
\subsection{Local existence} We solve the Dirichlet problem for a smooth Riemannian domain  $B \subset M$. We will  motivate the construction by  first considering  the case $X=\RR$ (cf. \cite [Section 2.2]{korevaar-schoen1}).
Fix $\phi \in H^1(B, X)$ and consider the space
\[
H^1_\phi(B, X)=\{f \in  H^1(B, X): f-\phi \in H^1_0(B, X)     \}
\]
Let 
\[
E_0=\inf \{ E(f): f \in  H^1_\phi(B, X) \}.
\]
By the parallelogram identity
\[
2\int_B |d\frac{f+v}{2}|^2\star1+2\int_B |d\frac{f-v}{2}|^2\star1=\int_B |df|^2\star1+\int_B |dv|^2\star1
\]
Take a minimizing sequence $f_i$ and apply the previous equality for $f=f_i$, $v=f_j$.
This implies that 
\begin{eqnarray*}
2\int_B |d\frac{f_i-f_j}{2}|^2\star1
&=& \int_B |df_i|^2\star1+\int_B |df_j|^2\star1-2\int_B |d\frac{f_i+f_j}{2}|^2\star1 \\
&\leq& 2E_0+2\epsilon_i-2E_0=2\epsilon_i.
\end{eqnarray*}
Hence
\begin{eqnarray}\label{ineqconvuoi}
\lim \int_B |d\frac{f_i-f_j}{2}|^2\star1=0.
\end{eqnarray}
By the Poincare inequality 
\begin{eqnarray}
\lim \int_B |\frac{f_i-f_j}{2}|^2\star1=0
\end{eqnarray}
hence
\[
\lim_{i \rightarrow \infty} f_i=f \ \mbox{in} \  H^1_\phi(B, X) \ \mbox{and} \ E(f)=E_0.
\]

Now assume $X$ is an NPC space.  Korevaar-Schoen \cite{korevaar-schoen1} showed that the energy density  makes sense by taking difference quotients.  For the purpose of these lectures, if $X$ is a locally finite Euclidean building, then we can locally isometricaly embed it in a Euclidean space of  high dimension. Then, we can define the energy density of the map to the building equal to the energy density of the map considered as a map to the Euclidean space. In fact, this was the original point of view taken in \cite{gromov-schoen}. The more general theory developed later  in \cite{korevaar-schoen1} and \cite{korevaar-schoen2}.

With this, 
we argue as above replacing  the parallelogram identity by the quadrilateral inequality.  Indeed, for $f,v \in H^1_\phi(B,X)$,  define $w(x)=(1-t)f(x)+tv(x)$.  Then (\ref{menelaus}) with $t=\frac{1}{2}$ implies 
\begin{align*}
2 d^2(w(x), w(y)) &\leq d^2(f(x),f(y))+ d^2(v(x),v(y))-\frac{1}{2} (d(f(y),v(y))-d(f(x),v(x))^2
\end{align*}
which then implies 
\[
2E^w \leq E^f+E^v- \frac{1}{2} \int_B |\nabla d(f,v)|^2\star1
\]
Take a minimizing sequence $f_i$ and apply the previous inequality with $f=f_i$ and $v=v_i$ to conclude (cf.~(\ref{ineqconvuoi}))
\[
\lim_{i,j \rightarrow \infty} \int_B |\nabla d(f_i,f_j)|^2  \star1\rightarrow 0.
\]
By the Poincare inequality, $f_i$ is a Cauchy sequence in $(Y,d_Y)$  and converges to a map  which is minimizing by the lower semicontinuity of energy \cite[Theorem 1.6.1]{korevaar-schoen1}.

\subsection{Basic Regularity result of Gromov-Schoen and Korevaar-Schoen}

This is the analogue of (\ref{moser}) without using the PDE.
\begin{theorem}\label{lipschitz} If $f \in H^1(B, X)$ is a harmonic map, then $f$ is locally Lipschitz.  More precisely,  for any  $B' \subset \subset B$, there exists a constant $C$ only depending on the metric on $B'$ and the distance of $B'$ to $\partial B$  such that
\[
\sup_{B'}|df|^2 \leq c \int_B |df|^2\star1.
\]
\end{theorem}

\subsection{Equivariant maps}
Let $\rho: \pi_1(M) \rightarrow \mathsf{Isom}(X)$ be a homomorphism.
A map
\[
v: \tilde M \rightarrow X
\]
is called a $\rho$-equivariant map, if
\[
v(\gamma x)=\rho(\gamma)v(x).
\]
 Since $|dv|^2$ is $ \pi_1(M)$-invariant,  it descends to a function on $M$. Define:
\[
E(v)=\int_M |d  v|^2\star1.
\]
If $v$ descends to a map to $M/\rho(\pi_1(M))$ this agrees with our previous definition.

\subsection{Existence of $\rho$-equivariant locally Lipschitz maps}
Let $(\mathcal M, \nu)$ be a probability space, $X$ an NPC-space and $f \in L^2(\mathcal M, X)$.

\begin{lemma} \label{com}
There exists a unique point $Q_{f,\nu}$ that   minimizes the integral 
\[
I_{f, \nu}(Q):=\int_{\mathcal M}d^2 (f(m), Q) d\nu(m)\ \ \forall Q \in X.
\]
We call $Q_{f,\nu}$ the center of mass.
\end{lemma}

\begin{proof}
Let $\{Q_i\}$ be a minimizing sequence and let $Q_{ij}=\frac{1}{2}Q_i+\frac{1}{2}Q_j$.
By (\ref{tricomp}) with $t=\frac{1}{2}$, 
$$d^2(f(x),Q_{ij}) \leq \frac{1}{2} d^2(f(x),Q_i)+\frac{1}{2} d^2(f(x),Q_j)-\frac{1}{4} d^2(Q_i,Q_j).
$$
Integrating, we obtain
$$
I_{f,\nu}(Q_{ij}) \leq \frac{1}{2} I_{f,\nu}(Q_i) + \frac{1}{2} I_{f,\nu}(Q_j) -\frac{1}{2}d^2(Q_i,Q_j).
$$
Thus, $d^2(Q_i,Q_j)$ is a Cauchy sequence.  We conclude that any minimizing sequence is a Cauchy sequence and converges to a minimizing element.
\end{proof}
\begin{lemma}There exists a locally Lipschitz $\rho$-equivariant map $\tilde f: \tilde M \rightarrow X$. If $X$ is smooth, then $\tilde f$ can be chosen to be smooth.
\end{lemma}
%
%
%
%
%

\begin{proof}

Let  $Q_0:=Q_{f,\mu_0}$ (resp.~$Q_1:=Q_{f,\mu_1}$) be the center of mass for  the  function $f \in L^2(\mathcal M, X)$ and the probability space $(\calM,\mu_0)$  (resp.~$(\calM, \mu_1))$. Let $Q_t=(1-t)Q_0+tQ_1$.
By the minimizing property of $Q_0$ and $Q_1$, 
\begin{eqnarray*}
\lefteqn{\int d^2(f,Q_0) + d^2(f,Q_1) \ d\mu_0}
\\
& = & \int d^2(f,Q_0) d\mu_0 + \int d^2(f,Q_1) d\mu_1+\int d^2(f,Q_1) (d\mu_0-d\mu_1) \label{abc}
\nonumber \\
& \leq &   2\int d^2(f,Q_{1/2})\, d\mu_0   +\int \left( d^2(f,Q_1) -  d^2(f,Q_{1/2}) \right)(d\mu_0-d\mu_1)
\nonumber
\\
& \leq & \int d^2(f,Q_0)+d^2(f,Q_1) -\frac{1}{4} d^2(Q_0,Q_1)  \, d\mu_0
\\
& & \ \  +\int \left( d^2(f,Q_1) -  d^2(f,Q_{1/2}) \right)(d\mu_0-d\mu_1).
\end{eqnarray*}
The last inequality is by  triangle comparison.
Consequently,
\begin{equation} \label{se}
d^2(Q_0,Q_1) \leq 4\int \left(d^2(f,Q_1)-d^2(f,Q_\frac{1}{2}) \right)(d\mu_0-d\mu_1).
\end{equation}

For each $x \in \tilde M$, 
let 
$$
d\mu_x=\frac{d\mathsf{vol} \mres B_1(x)}{V(x)}, \ \ V(x)=\mathsf{vol}(B_1(x)).
$$
where $\mathsf{vol}=\star1$ is the volume form of $\tilde M$.  
Since $\mu_x$ is only dependent on the metric of $M$, $x \mapsto \mu_x$ is invariant under the isometric action of $\pi_1(M)$.  Furthermore, 
\begin{align*}
\int |d\mu_{x_0}-d\mu_{x_1}| 
& = \int \left| \frac{1}{V(x_0)} \chi_{B_1(x_0)} - \frac{1}{V(x_0)} \chi_{B_1(x_0)} \right| \star1
\\
& \leq C\rho(x_0,x_1)
\end{align*}
where $\rho$ denotes the distance function on $M$.

Let $M_0$ be a fundamental domain.  Let $f(M_0)=P$ and extend equivariantly to 
$
f:\tilde M \rightarrow X.
$    
For simplicity, assume $M_0$ is compact.  Then there exists a constant $L$ such that 
\[
d(f(x_0), f(x_1)) \leq L  \ \mbox{whenever} \  \rho(x_0,x_1)<2.
\]
Thus, for  $\rho(x_0,x_1)<1$ and  in the support of $|d\mu_{x_0}-d\mu_{x_1}|$, 
\[
d^2(f,Q_1) - d^2(f,Q_{1/2}) \leq d^2(f,Q_1) + d^2(f,Q_t) \leq 2L^2.
\]

Define
\[
\tilde f: \tilde M \rightarrow X, \ \ \tilde f(x)=Q_{f,\mu_x}
\]
The $\pi_1(M)$-invariance of $\mu_x$ and the $\rho$-equivariance of $f$ imply the  $\rho$-equivariance of $\tilde f$.  
Apply (\ref{se}) with $\calM=M$, $\mu_0=\mu_{x_0}$ and $\mu_1=\mu_{x_1}$ to obtain
\begin{align*}
d^2(\tilde f(x_0), \tilde f(x_1)) & \leq   2L^2 \int |d\mu_{x_0}-d\mu_{x_1} |\leq  2L^2 C \rho(x_0,x_1).  \qedhere
\end{align*}
\end{proof}

\subsection{The boundary at infinity} A good reference is \cite{bridson-haefliger}.
Suppose  $X$ is an NPC-space.  Two geodesic rays $c, c' : [0, \infty) \rightarrow X$ are said to be asymptotic if there exists a constant $K$ such that $d(c(t),c'(t))< K$ for all $t > 0$. The set $\partial X$ of boundary points of $X$ (which we shall also call the points at infinity) is the set of equivalence classes of geodesic rays, two geodesic rays being equivalent if and only if they are asymptotic. We denote $\bar X = X \cup \partial X$. 
Notice that the images of two asymptotic geodesic rays under any isometry  of $X$ are again asymptotic geodesic rays, and hence any isometry extends to give a bijection of $\bar X$. The next proposition is \cite[Proposition 8.2]{bridson-haefliger}.

\begin{proposition}If $X$ is an NPC-space and $c : [0, \infty) \rightarrow X$ is a geodesic ray starting from $P$, then for every point $P_1 \in X$ there is a unique geodesic ray  which starts from $P_1$ and is asymptotic to $c$.
\end{proposition}

The topology of $\bar X$ is defined as follows:
A sequence of points $P_i$ converges to a point $P^* \in \partial X$ if and only if the geodesics joining $P_0$ to $P_i$ converge (uniformly on compact subsets) to the geodesic ray that starts from $P_0$ and belongs to the class of $P^*$.

\begin{example} If $X$ is a complete $n$-dimensional Riemannian manifold of non-positive sectional curvature, then $\partial X$ is homeomorphic to $S^{n-1}$.  Indeed, given
a base point $P_0$, we can obtain a homeomorphism by considering the map  which associates to each unit vector $V$ tangent to $X$ at $P_0$ the class of the geodesic ray $c$ starting at $P_0$ with velocity vector $V.$ In particular, if $X$  is the   $n$-dimensional hyperbolic space, then $\bar X$ is homeomorphic to the $n$-dimensional ball in $\RR^n$. If $X$ is a locally compact Euclidean building, then $\partial X$ is a compact spherical building (cf. \cite[Proposition 4.2.1]{kleiner}).
\end{example}
%

\begin{lemma}\label{equlaconv}
If 
$P_i$ is a sequence in $X$ with 
$\lim P_i=P^* \in \partial X$ and if $Q_i$ is another sequence in $X$ with $d(P_i, Q_i) \leq C$ independently of $i$, then  $\lim Q_i=P^*$. 
\end{lemma}

\begin{proof}
Fix $P_0 \in X$.  Let $\gamma:[0,\infty) \rightarrow \infty$ be an arclength parameterized geodesic ray  in the equivalence class $P^*$ with $\gamma(0)=P_0$.  Let $t_i=d(P_0,P_i)$ (resp.~$\tau_i=d(P_0,Q_i)$) and let $\gamma_i:[0,t_i] \rightarrow X$ (resp.~$\hat \gamma_i:[0,\tau_i] \rightarrow X$) be the arclength parameterized geodesic segment connecting $P_0$ and $P_i$ (resp.~$Q_i)$.   By the triangle inequality, 
\[
d(\hat \gamma_i(t_i), \hat \gamma_i(\tau_i))=|t_i-\tau_i|=|d(P_0,P_i)-d(P_0,Q_i)| \leq d(P_i,Q_i) \leq C.
\]  
Thus, assuming $t \leq t_i \leq \tau_i$, the NPC condition implies 
\[
d(\hat \gamma_i(t), \gamma_i(t)) \leq  \frac{t}{t_i} d(\hat \gamma_i(t_i), \gamma_i(t_i)) \leq \frac{t}{t_i}\Big(d(\hat \gamma_i(t_i), \hat \gamma_i(\tau_i))+ d(\hat \gamma_i(\tau_i), \gamma_i(t_i)) \Big) \leq 
 \frac{2Ct}{t_i}.
\]
Similarly, assuming $t \leq \tau_i < t_i$,
\[
d(\hat \gamma_i(t), \gamma_i(t)) \leq  \frac{2Ct}{\tau_i}.
\]
Thus, for $t \leq \min\{t_i, \tau_i\}$, 
\[
d(\hat \gamma_i(t), \gamma(t)) \leq  d(\hat \gamma_i(t), \gamma_i(t)) + d(\gamma_i(t), \gamma(t))
\leq  \frac{2Ct}{\max\{t_i, \tau_i\}}+d(\gamma_i(t), \gamma(t)).
\]
Fix $T_0>\infty$.  The assumption that  $\lim P_i =P^*$ implies that $t_i, \tau_i \rightarrow \infty$ and the geodesics $\gamma_i$ converge uniformly to $\gamma$ in $[0,T_0]$. Thus, $\hat \gamma_i$ also converge uniformly to $\gamma$ in $[0,T_0]$.
\end{proof}

\begin{lemma}The stabilizer of a point at infinity is contained in a parabolic subgroup. So if the image of $\rho$ is Zariski dense it cannot fix a point at infinity.
\end{lemma}

\subsection{Global existence result}
We prove  existence of  equivariant harmonic maps \cite[Theorem 7.1]{gromov-schoen}.

\begin{theorem}\label{existence}
Let $X$ be a locally compact NPC space. Assume that the image of $\rho$ doesn't fix a point in $\partial X$ and that  there exists a Lipschitz equivariant map 
 $v: \tilde M \rightarrow X$ with finite energy. Then there is a Lipschitz equivariant map $f$ of least energy and the restriction of $f$ to a small ball about any point is minimizing.
\end{theorem}
\begin{proof}
Let $E_0$ denote the infimum of the energy taken over all Lipschitz equivariant maps. Let $v_i$ be a sequence of Lipschitz equivariant maps with $E(v_i) \rightarrow E_0$. Let $B$ be a ball in $\tilde M$ such that $\gamma(B) \cap B= \emptyset$ for all $\gamma \in \pi_1(M)$. We may then construct a new minimizing sequence $\bar v_i$, by replacing $v_i$ with the solution to the Dirichlet problem on each $\gamma(B)$. Clearly $\bar v_i$ is also a minimizing sequence.

 On a compact subset of $B$, the sequence $\bar v_i$ is uniformly Lipschitz by Theorem~\ref{lipschitz}. It follows that a subsequence of $\bar v_i$ converges uniformly on compact subsets of $B$ to a map into $\bar X$ which either maps into $X$ or maps to a single point $P^* \in \partial X$. We  exclude the second possibility as follows. Let $x_0 \in \tilde M$ be the center of the chosen ball $B$. Let $C$ be any smooth embedded curve from $x_0$ to $\gamma(x_0)$. An elementary argument using Fubini's theorem shows that $C$ may be chosen so that the energy of the restriction of each map $\bar v_i$ to $C$ is uniformly bounded. Therefore the length of the curve $\bar v_i(C)$ is uniformly bounded, and in particular 
 $d(v_i(x_0), \rho(\gamma)v_i((x_0)) \leq C$.
Lemma~\ref{equlaconv} implies $\lim \rho(\gamma)v_i(x_0)=P^*$, and hence $\rho(\gamma)P^*=P^*$ for all $\gamma$.
 This is a contradiction.
Therefore we may assume that $\bar v_i$ converges uniformly on compact subsets of $B$. 

From~(\ref{ineqconvuoi}) as before, we have
\[
\int_K|\nabla d(\bar v_i,\bar v_j)|^2\star1 \rightarrow 0
\]
for any compact set $K \subset \tilde M$. Since $\bar v_i$ converges uniformly on compact subsets of $B$, the function 
$d(\bar v_i,v)$ is uniformly bounded there. It then follows from Poincare type inequalities that
\[
\int_K d(\bar v_i,\bar v_j)^2\star1 \rightarrow 0.
\]
In particular, the sequence $\bar v_i \rightarrow f$ which is  a minimizer by lower semicontinuity of energy. The local minimizing property of $f$ follows. This completes the proof.
\end{proof}

\newpage

\section{Pluriharmonic maps and the Siu-Sampson Formula}
\label{sec:siu}

\subsection{Introduction: Bochner methods for harmonic maps} In this section, we discuss the Bochner formulas of Siu \cite{siu} and Sampson \cite{sampson}.  Our exposition closely follows  the approach  of \cite{liu-yang}.   We also present a variation of these formulas inspired by the  work of Mochizuki \cite{mochizuki-memoirs}.  Lastly, we sketch the existence of  pluriharmonic maps into Euclidean buildings. 

\subsection{Pluriharmonic maps from  K\"ahler manifolds to Riemannian manifolds}  \label{KtoR}
Let $(M,\omega, J)$ be a K\"ahler manifold along with its K\"ahler form and complex structure.  
Let 
\[
TM \otimes \mathbb C = T^{(1,0)}M \oplus T^{(0,1)}M
\]
be its complexified tangent bundle decomposed into the $\pm \sqrt{-1}$-eigenspaces of $J$.
We can decompose $v \in TM \otimes \mathbb C$ into
\[
v=v^{1,0}+v^{0,1} \ \mbox{where} \ v^{1,0}=\frac{1}{2}(v-\sqrt{-1}Jv), \  v^{0,1}=\frac{1}{2}(v+\sqrt{-1}Jv).
\]
The cotangent space $T^*M$ has a complex structure still denoted $J$ defined by $J \alpha=\alpha \circ J$.  Accordingly, we have an analogous decomposition
\[
T^*M \otimes \CC=T^{*(1,0)}M \oplus T^{*(0,1)}M.
\]

Let $(N,h)$ be a Riemannian manifold and $TN \otimes \mathbb C$ its complexified tangent bundle.  For a smooth map $f:M \rightarrow N$, let
\begin{equation} \label{E}
E:=f^*(TN \otimes \CC).
\end{equation}
Extending complex linearly,   $df:TM \rightarrow TN$ gives rise to a map
$df:  TM \otimes \CC \rightarrow TN \otimes \CC.$
Denote by $\Omega^{p,q}(E)$,  the space of   $E$-valued $(p,q)$-forms.  Define
\[
d' f:=\frac{1}{2}(df- \sqrt{-1}\,df \circ J) \in \Omega^{1,0}(E),\ \  \ \ d'' f:=\frac{1}{2}(df+\sqrt{-1}\,df \circ J) \in \Omega^{0,1}(E).
\]
We have that
\[
df=d' f+d'' f \ \ \ \ \ \ J df= df \circ J=-\sqrt{-1}\,(d' f-d'' f).
\]
For local coordinates $(y^i)$ of $N$, let $\frac{\partial}{\partial f^i} = \frac{\partial}{\partial y^i} \circ f$.  Then
\begin{eqnarray*}
d' f=d' f^i \frac{\partial}{\partial f^i} 
& & 
d'' f=d'' f^i\frac{\partial}{\partial f^i} 
\\
\overline{d' f} = d'' f \  & &  \overline{d'' f}=d' f.
\end{eqnarray*}
Similarly, 
we can decompose the pullback of the Levi-Civita connection (cf.~Section~\ref{sec:smooth harmonic maps}) as
\[
\nabla=\nabla'+\nabla''
\]
where 
\[
\nabla':C^{\infty}(E) \rightarrow \Omega^{1,0}(E), \ \ \ \nabla'':C^{\infty}(E) \rightarrow \Omega^{0,1}(E).
\]
In turn, $\nabla'$ and $\nabla''$ induce differential operators
\[
d'_E: \Omega^{p,q}(E) \rightarrow \Omega^{p+1,q}(E), \ \ \
d''_E: \Omega^{p,q}(E) \rightarrow \Omega^{p,q+1}(E)
\]
where
\begin{eqnarray*}
d'_E(\phi \otimes s) & = &  d'\phi \otimes s+(-1)^{p+q}\phi \otimes \nabla'_Es
\\
d''_E(\phi \otimes s) & = &  d''\phi \otimes s+(-1)^{p+q}\phi \otimes \nabla''_Es.
\end{eqnarray*}
A straightforward calculation implies that
\begin{equation} \label{easycommute'}
d'_{E} d'' f = - d''_{E} d' f,\ \ \
d'_{E} d' f = 0,  \ \ \  d''_{E} d'' f = 0.
\end{equation}

\begin{lemma}\label{trace}
\[
\tau (f)=2i  \star \left(\frac{ \omega^{n-1}}{(n-1)!} \wedge d'_{E} d'' f \right).
\]
\end{lemma}
\begin{proof}
We claim
\[
\star \alpha = \frac{\omega^{n-1}}{(n-1)!}\wedge J\alpha, \ \ \ \forall \alpha \in \Omega^1(M, \RR).
\]
To check the claim, use  normal coordinates  $(z^i=x^i+ \sqrt{-1} y^i)$  at a point $x \in M$.  For $\alpha=dx^i$ or  $\alpha=dy^i$, we have \[
dx^i \wedge \frac{\omega^{n-1}}{(n-1)!} \wedge Jdx^i  
= dx^i \wedge dy^i  \wedge \frac{\omega^{n-1}}{(n-1)!} 
= \frac{\sqrt{-1}}{2} dz^i \wedge d\bar z^i \wedge \frac{\omega^{n-1}}{(n-1)!} = \frac{\omega^n}{n!}
\]
and
\[
dy^i \wedge \frac{\omega^{n-1}}{(n-1)!} \wedge Jdy^i  
= -dy^i \wedge dx^i  \wedge \frac{\omega^{n-1}}{(n-1)!} 
= \frac{\sqrt{-1}}{2} dz^i \wedge d\bar z^i \wedge \frac{\omega^{n-1}}{(n-1)!} = \frac{\omega^n}{n!}.
\]
The claim follows by linearity.

Next, note that  
\begin{eqnarray*}
J d'f & = &\frac{1}{2} (df \circ J + \sqrt{-1} df)= \frac{\sqrt{-1}}{2}(df-\sqrt{-1} df \circ J) =\sqrt{-1}d'f
\\
Jd''f & = &  \frac{1}{2} (df \circ J-\sqrt{-1}df) = \frac{\sqrt{-1}}{2} (df + \sqrt{-1} df \circ J) = - \sqrt{-1} d''f,
\end{eqnarray*}
which implies $Jdf=J d'f+Jd''f=\sqrt{-1}(d'f-d''f)$.
Applying the claim for $\alpha=df$, we use the fact that  $d\omega=0$ to obtain
\begin{eqnarray*}
\tau (f)&=&-d_\nabla^{\star} df \\
&=&\star d_\nabla( \star df)\\
&=&\star d_\nabla\left (\frac{\omega^{n-1}}{(n-1)!}\wedge Jdf \right)\\
&=&\star d_\nabla \left( \frac{\omega^{n-1}}{(n-1)!} \wedge (\sqrt{-1}(d' f -d'' f)) \right)\\
&=&-\sqrt{-1} \star \left(\frac{ \omega^{n-1}}{(n-1)!} \wedge (d'_{E} d'' f -d''_{E} d' f) \right)\\
&=&-2\sqrt{-1}  \star \left(\frac{ \omega^{n-1}}{(n-1)!} \wedge d'_{E} d'' f \right).
\end{eqnarray*}
\end{proof}

\begin{definition}$f$ is called {\it pluriharmonic} $d'_{E}d'' f =0.$
\end{definition}

\begin{remark}  \label{holimpleishar'}
Lemma~\ref{trace} implies
\[
\mbox{pluriharmonic} \Longrightarrow \mbox{harmonic}.
\]
Note that holomorphic maps between K\"ahler manifolds are pluriharmonic, and thus harmonic.
\end{remark}

\subsection{Sampson's Bochner formula}

 \begin{theorem}[Sampson's Bochner formula, \cite{sampson}]
\label{sampsonbochner}
For a harmonic map $f: M \rightarrow N$ from a K\"{a}hler manifold $(M,g)$ to a Riemannian manifold $(N,h)$,
\begin{eqnarray*}
d'  d''\{d'' f,d'' f \}\wedge \frac{ \omega^{n-2}}{(n-2)!} \nonumber
& = &    4\left(\left|d'_E d'' f \right|^2 +Q_0\right) \frac{ \omega^n}{n!} 
\end{eqnarray*}
where $\{ \cdot, \cdot\}$ is given in Definition~\ref{bracket} below and 
\[
Q_0=-2g^{\alpha \bar \delta} g^{\gamma \bar \beta} R_{ijkl}
\frac{\partial f^i}{\partial z^\alpha}  \frac{\partial f^k}{\partial \bar z^\beta}  \frac{\partial f^j}{\partial z^\gamma}\frac{\partial f^l}{\partial \bar z^\delta}\]
in local coordinates $(z^\alpha)$ of $M$ and $(y^i)$ of $N$. 
\end{theorem}

\begin{proof}
Combine Lemma~\ref{brac'}, Lemma~\ref{aux'} and Lemma~\ref{tree'} below.
\end{proof}

\begin{definition} \label{bracket}
Let $\{s^i\}$ be a  local frame of $E$.  For 
\begin{eqnarray*}
\psi  =  \psi_i \otimes s^i \in \Omega^{p,q}(E) \ \mbox{and }  \
 \xi  =  \xi_i \otimes s^i \in \Omega^{p',q'}(E)
 \end{eqnarray*} 
we set
\[ 
\{\psi, \xi \}= \langle s^i, s^j \rangle
 \psi_i \wedge \bar \xi_j
 \in \Omega^{p+q',q+p'}
\] 
where $ \langle \cdot, \cdot \rangle$ is  the complex-linear extention of the Riemannian metric on $E$.
\end{definition}

\begin{lemma}\label{brac'}
For any smooth map $f : M \rightarrow N $ from a  K\"ahler manifold to a Riemannian manifold, we have

\[
d'  d''\{d'' f,d'' f \}\wedge \frac{ \omega^{n-2}}{(n-2)!}=\left(-\{d'_{E}d'' f,  d'_{E}d'' f \}+ \{d'' f, R_E^{(1,1)}(d'' f) \}\right)\wedge \frac{ \omega^{n-2}}{(n-2)!}
\]
where
\[R_E^{(1,1)}
=(d'_E d''_E + d''_E d'_E)
\]
is the $(1,1)$-part of the curvature $R_E=d_E^2$.
\end{lemma}
\begin{proof}
Repeatedly using the fact that $d''_{E} d'' f=0$ (cf.~(\ref{easycommute'})), 
\begin{eqnarray*}
d'  d''\{d'' f,d'' f \}
& =&-\{d'_{E}d'' f,  d'_{E}d'' f \}+ \{d'' f, d''_{E} d'_{E}d'' f \}
\\
& =&-\{d'_{E}d'' f,  d'_{E}d'' f \}+ \{d'' f, (d''_{E} d'_{E} +   d'_{E} d''_{E}+ {d''_E}^2) d'' f  \}. 
\end{eqnarray*}
Since $\{d'' f, d'_E d'_{E} d'' f\} \wedge \frac{ \omega^{n-2}}{(n-2)!}$ is an 
$(n-1,n+1)$-form and hence zero  for dimensional reasons, we can complete the square to obtain
\begin{eqnarray}
d'  d''\{d'' f, d'' f \}\wedge \frac{ \omega^{n-2}}{(n-2)!}
& =&\left(-\{d'_{E}d'' f,  d'_{E}d'' f \}+ \{d'' f, (d'_{E} + d''_{E} )^2 d'' f\}\right) \wedge \frac{ \omega^{n-2}}{(n-2)!} \nonumber\\
& =&\left(-\{d'_{E}d'' f,  d'_{E}d'' f \}+ \{d'' f, R_E^{(1,1)}( d'' f) \} \right)\wedge \frac{ \omega^{n-2}}{(n-2)!} \nonumber
\end{eqnarray}
which proves the first equation.
\end{proof}

\begin{lemma} \label{stc}
For any $E$-valued $(1,1)$-form $\phi$ on $M$,  
\[
-\{\phi, \phi\}\wedge \frac{ \omega^{n-2}}{(n-2)!}=4(|\phi|^2-|\mathsf{Trace}_\omega\phi|^2)\frac{ \omega^n}{n!}.
\]
\end{lemma}

\begin{proof}
Let $(z^p)$ be normal coordinates at $x \in M$ and let  $\phi_{p\bar q} dz^p \wedge d\bar z^q$.
  At $x$,
\[
\frac{\omega^n}{n!} = \left( \frac{\sqrt{-1}}{2} \right)^2 \left( \sum_p dz^p \wedge d\bar z^p \right)^2 \wedge \frac{\omega^{n-2}}{(n-2)!}.
\]
For  $p, q$ such that $p \neq q$, 
\[
s \neq p \mbox{ or } t \neq q \ \Rightarrow \ dz^p \wedge d\bar z^q \wedge d\bar z^s \wedge dz^t \wedge \left( \sum_j dz^j \wedge d\bar z^j \right)^{n-2} =0.
\]
For  $p=q$, 
\[
s \neq t\ \Rightarrow \ dz^p \wedge d\bar z^q \wedge d\bar z^s \wedge dz^t \wedge \left( \sum_j dz^j \wedge d\bar z^j \right)^{n-2} =0.
\]
Furthermore,
\begin{eqnarray*}
\phi_{p\bar q} \overline{\phi_{p \bar q}}  dz^p \wedge d\bar z^q \wedge d\bar z^p \wedge dz^q & = & \phi_{p\bar q} \overline{\phi_{p \bar q}}  dz^p \wedge d\bar z^p \wedge dz^q \wedge d\bar z^q
\\
\phi_{p\bar p} \overline{\phi_{q \bar q}}  dz^p \wedge d\bar z^p \wedge d\bar z^q \wedge dz^q & = & -\phi_{p\bar p} \overline{\phi_{q \bar q}}   dz^p \wedge d\bar z^p \wedge dz^q \wedge d\bar z^q.
\end{eqnarray*}
Thus,
\begin{eqnarray*}
 \left( \frac{\sqrt{-1}}{2} \right)^2 \{ \phi, \phi\} \wedge \frac{\omega^{n-2}}{(n-2)!} 
 & = &
 \left( \frac{\sqrt{-1}}{2} \right)^2 \left( \sum_{p,q,s,t} \phi_{p\bar q} \overline{\phi_{s\bar t}}
 dz^p \wedge d\bar z^q \wedge d\bar z^s \wedge dz^t  \right) \wedge 
 \frac{\omega^{n-2}}{(n-2)!}
  \\
  &= & 
  \sum_{p \neq q} \left( |\phi_{p\bar q}|^2-\phi_{p\bar p} \phi_{q \bar q}  \right) \frac{\omega^n}{n!}
  \\
   &= & 
  \sum_{p,q} \left( |\phi_{p\bar q}|^2-\phi_{p\bar p} \phi_{q \bar q}  \right) \frac{\omega^n}{n!}
  \\
  & = & \left( |\phi|^2- |\mathsf{trace}_\omega \phi|^2 \right)
   \frac{\omega^n}{n!}. \ \ \ \ \ \ \ \   \ \ \ \ \ \ \ \ \qedhere
\end{eqnarray*}
\end{proof}
\begin{lemma}\label{aux'}
For any harmonic map $f : M \rightarrow N $ from a   K\"ahler manifold to a Riemannian manifold, we have
\begin{eqnarray*}
-\{d'_{E}d'' f,  d'_{E}d'' f \} \wedge \frac{ \omega^{n-2}}{(n-2)!} & = &  
4\left|d'_{E}d'' f \right|^2 \frac{ \omega^n}{n!} .
\end{eqnarray*}
\end{lemma}

\begin{proof}
We apply Lemma~\ref{stc} with $\phi=d'_{E}d'' f$.   Since $f$ is harmonic, $\mbox{Tr}_{\omega} d'_{E}d'' f=0$ by  Lemma~\ref{trace}.
\end{proof}

\begin{lemma}\label{tree'}
For a harmonic map $f : M \rightarrow N $ from a  K\"ahler manifold to a Hermitian-negative Riemannian manifold, we have
\[
\{d''  f, R_E^{(1,1)}  (d''  f) \} \wedge \frac{\omega^{n-2}}{(n-2)!}= -2\, R_{i j k l}  d' f^i  \wedge d'' f^k  \wedge d' f^j \wedge  d'' f^l \wedge \frac{\omega^n}{n!}
\]
where $R_E^{(1,1)} $ is defined in Lemma~\ref{brac'}.
\end{lemma}
\begin{proof}
Let $(z^\alpha)$ (resp.~$(y^i)$) be normal coordinates at a point $x \in M$ (resp.~$f(x) \in N$).  Then
\begin{eqnarray*}
\nabla_{\frac{\partial}{\partial \bar z^\gamma}}  \nabla_{\frac{\partial}{\partial  z^\beta}} \frac{\partial}{\partial f^j} & = & \nabla_{\frac{\partial}{\partial \bar z^\gamma}}  \left( \frac{\partial f^k}{\partial z^\beta} \nabla_{\frac{\partial}{\partial  f^k}} \frac{\partial}{\partial f^j} \right)
\\
& = & \frac{\partial f^k}{\partial z^\beta}  \frac{\partial f^l}{\partial \bar z^\gamma} \nabla_{\frac{\partial}{\partial f^l}}  \nabla_{\frac{\partial}{\partial  f^k}} \frac{\partial}{\partial f^j} 
\end{eqnarray*}
and
\begin{eqnarray*}
d''_E d'_E d''f & = & d''_E d'_E \left( \frac{\partial f^j}{\partial \bar z^\alpha} d\bar z^\alpha \otimes \frac{\partial}{\partial f^j} \right) \\
& = & d'' d' \left( \frac{\partial f^j}{\partial \bar z^\alpha} d\bar z^\alpha \right) \otimes \frac{\partial}{\partial f^j} - \frac{\partial f^j}{\partial \bar z^\alpha} d\bar z^\alpha \wedge d\bar z^\gamma \wedge dz^\beta \otimes \nabla_{\frac{\partial}{\partial \bar z^\gamma}}  \nabla_{\frac{\partial}{\partial  z^\beta}} \frac{\partial}{\partial f^j} 
\\
& = & d'' d' \left( \frac{\partial f^j}{\partial \bar z^\alpha} d\bar z^\alpha \right) \otimes \frac{\partial}{\partial f^j} +\frac{\partial f^j}{\partial \bar z^\alpha} \frac{\partial f^k}{\partial z^\beta}  \frac{\partial f^l}{\partial \bar z^\gamma}  d\bar z^\alpha \wedge d z^\beta \wedge d\bar z^\gamma  \otimes 
\nabla_{\frac{\partial}{\partial f^l}}  \nabla_{\frac{\partial}{\partial  f^k}} \frac{\partial}{\partial f^j}.
\end{eqnarray*}
Similarly,
\begin{eqnarray*}
d'_E d''_E d''f  & = & d' d'' \left( \frac{\partial f^j}{\partial \bar z^\alpha} d\bar z^\alpha \right) \otimes \frac{\partial}{\partial f^j} - \frac{\partial f^j}{\partial \bar z^\alpha} \frac{\partial f^k}{\partial z^\beta}  \frac{\partial f^l}{\partial \bar z^\gamma}  d\bar z^\alpha \wedge d z^\beta \wedge  d\bar z^\gamma \otimes 
\nabla_{\frac{\partial}{\partial f^k}}  \nabla_{\frac{\partial}{\partial  f^l}} \frac{\partial}{\partial f^j}.
\end{eqnarray*}
Combining the above two equalities,
\begin{eqnarray*}
R_E^{(1,1)} (d''f) & = &   \frac{\partial f^j}{\partial \bar z^\alpha} \frac{\partial f^k}{\partial z^\beta}  \frac{\partial f^l}{\partial \bar z^\gamma}  d\bar z^\alpha \wedge dz^\beta \wedge d \bar z^\gamma  \otimes R^s_{\, jkl} \frac{\partial}{\partial f^s}.
\end{eqnarray*}
We compute
 \begin{eqnarray*}\label{dminor}
\{d''  f, R_E^{(1,1)}  (d''  f) \} &=&  R_{ijkl} \frac{\partial f^i}{\partial \bar z^\delta}  \frac{\partial f^j}{\partial z^\alpha} \frac{\partial f^k}{\partial \bar z^\beta}  \frac{\partial f^l}{\partial z^\gamma}  d\bar z^\delta \wedge dz^\alpha \wedge d\bar z^\beta \wedge dz^\gamma 
\\
&=&  R_{jilk}  \frac{\partial f^j}{\partial z^\alpha}  \frac{\partial f^i}{\partial \bar z^\delta}  \frac{\partial f^l}{\partial z^\gamma}\frac{\partial f^k}{\partial \bar z^\beta}   dz^\alpha \wedge d\bar z^\delta \wedge  dz^\gamma  \wedge d\bar z^\beta 
\\
&=&  (-R_{jlki}+R_{jkli}) \frac{\partial f^j}{\partial z^\alpha}  \frac{\partial f^i}{\partial \bar z^\delta} \frac{\partial f^l}{\partial z^\gamma} \frac{\partial f^k}{\partial \bar z^\beta}   dz^\alpha \wedge d\bar z^\delta \wedge  dz^\gamma  \wedge d\bar z^\beta.
\end{eqnarray*}
Since
\begin{eqnarray*}
\lefteqn{R_{jkli} \frac{\partial f^j}{\partial z^\alpha}  \frac{\partial f^i}{\partial \bar z^\delta} \frac{\partial f^l}{\partial z^\gamma} \frac{\partial f^k}{\partial \bar z^\beta}   dz^\alpha \wedge d\bar z^\delta \wedge  dz^\gamma  \wedge d\bar z^\beta \wedge \frac{\omega^{n-2}}{(n-2)!}
}\\
& = & R_{jkli} \frac{\partial f^j}{\partial z^\alpha}  \frac{\partial f^i}{\partial \bar z^\delta} \frac{\partial f^l}{\partial z^\gamma} \frac{\partial f^k}{\partial \bar z^\beta}   dz^\alpha \wedge d\bar z^\alpha \wedge  dz^\beta \wedge d\bar z^\beta \wedge \frac{\omega^{n-2}}{(n-2)!}
\\
&  & \ + R_{jkli} \frac{\partial f^j}{\partial z^\alpha}  \frac{\partial f^i}{\partial \bar z^\delta} \frac{\partial f^l}{\partial z^\gamma} \frac{\partial f^k}{\partial \bar z^\beta}   dz^\alpha \wedge d\bar z^\beta \wedge  dz^\beta  \wedge d\bar z^\alpha \wedge \frac{\omega^{n-2}}{(n-2)!}
\\
& = & 0,
 \end{eqnarray*}
we obtain
\begin{eqnarray*}
\{d''  f, R_E^{(1,1)}  (d''  f) \} \wedge \frac{\omega^{n-2}}{(n-2)!} 
& = &  -R_{jlki}\frac{\partial f^j}{\partial z^\alpha}  \frac{\partial f^i}{\partial \bar z^\alpha} \frac{\partial f^l}{\partial z^\beta} \frac{\partial f^k}{\partial \bar z^\beta}   dz^\alpha \wedge d\bar z^\alpha \wedge  dz^\beta  \wedge d\bar z^\beta\wedge \frac{\omega^{n-2}}{(n-2)!} 
\\
&  & \  -R_{jlki} \frac{\partial f^j}{\partial z^\alpha}  \frac{\partial f^i}{\partial \bar z^\beta} \frac{\partial f^l}{\partial z^\beta} \frac{\partial f^k}{\partial \bar z^\alpha}   dz^\alpha \wedge d\bar z^\beta \wedge  dz^\beta \wedge d\bar z^\alpha \wedge \frac{\omega^{n-2}}{(n-2)!}
\\
& = &   -2R_{jlki}\frac{\partial f^j}{\partial z^\alpha}  \frac{\partial f^i}{\partial \bar z^\alpha} \frac{\partial f^l}{\partial z^\beta} \frac{\partial f^k}{\partial \bar z^\beta}   dz^\alpha \wedge d\bar z^\alpha \wedge  dz^\beta  \wedge d\bar z^\beta\wedge \frac{\omega^{n-2}}{(n-2)!} 
\\
& = &   -2R_{jlki} d'f^j \wedge d''f^i \wedge d'f^l \wedge d''f^k \wedge \frac{\omega^{n-2}}{(n-2)!}.   \nonumber \qedhere
\end{eqnarray*}

 \end{proof}

\begin{definition}
A Riemannian manifold $N$ is said to be {\it Hermitian-negative} (resp. {\it strongly} {\it Hermitian negative}) if
\[
R_{ijkl} A^{i\bar l}A^{j \bar k} \leq 0 \ \  (\mbox{resp.~$<0$})
\]
for any Hermitian semi-positive matrix $A=\big( A^{i\bar l} \big)$.
\end{definition}

\begin{remark}
Locally symmetric spaces whose irreducible local factors are all non-compact or Euclidean type are Hermitian negative (cf.~\cite[Theorem 2]{sampson}).
\end{remark}

\begin{theorem}[Sampson] \label{sampsonmainthm}
 If
$f: M \rightarrow N$ is a harmonic map from a K\"ahler manifold into a Hermitian negative Riemannian manifold, then $f$ is pluriharmonic. \end{theorem}

\begin{proof}
Integrate Sampson's Bochner formula over $M$. Applying Stoke's theorem results in the left hand side being 0.  The two terms on the right hand side are non-negative pointwise,  hence they must be identically equal to 0.  In particular, $d'_{E} d'' f=0$; i.e.~$f$ is is pluriharmonic.  
\end{proof}

 \subsection{Maps between K\"ahler manifolds} 

Let $f:M \rightarrow N$ be a smooth map between K\"ahler manifolds.  By decomposing 
\[
TN \otimes \CC=T^{(1,0)}N \oplus T^{(0,1)}N
\]
we get the decomposition of $E:=f^{-1}(TN \otimes \CC)$
as
\[
E=E' \oplus E'' \ \mbox{ where } \ E':=f^{-1} (T^{(1,0)}N) , \ \ E'':= f^{-1} (T^{(0,1)}N).
\]
Denote by $\Omega^{p,q}(E)$, $\Omega^{p,q}(E')$ and $\Omega^{p,q}(E'')$  the space of   $E$-, $E'$- and $E''$-valued $(p,q)$-forms respectively.  
If $(w^i)$ are local holomorphic coordinates in $N$, then
$\{\frac{\partial}{\partial f^i}  :=  \frac{\partial}{\partial w^i} \circ f, \
\frac{\partial}{\partial \bar f^i} := \frac{\partial}{\partial \bar w^i} \circ f \}$ 
is a local  frame  of $E$.
If $d'f$, $d'f'$ are as in Section~\ref{KtoR}, then
\[
d'f=\partial f+\partial \bar f, \ \ \ \ \ d''f=\bar \partial f+\bar \partial  \bar f, \ \ \ \ \ df=d'f+d''f=\partial f+\partial \bar f+\bar \partial f+\bar \partial  \bar f.
\]
where
\begin{eqnarray*}
\partial f=\partial f^i \frac{\partial}{\partial f^i} 
& & 
\bar \partial f=\bar \partial f^i\frac{\partial}{\partial f^i} 
\\
\partial \bar f= \partial\bar f^i \frac{\partial}{\partial \bar f^i} 
& & 
\bar \partial \bar f=\bar \partial \bar f^i \frac{\partial}{\partial \bar f^i}\\
\overline{\partial f} = \bar \partial \bar f \  & &  \overline{\bar \partial f}=\partial \bar f.
\end{eqnarray*}
Analogously, $d_\nabla=d'_E+d''_E$ is decomposed into the induced operators $\partial_{E'}$, $\bar\partial_{E'}$, $\partial_{E''}$, $\bar\partial_{E''}$.


A straightforward calculation yields
\begin{eqnarray} \label{easycommute}
\partial_{E'} \bar \partial f = - \bar \partial_{E'} \partial f & &  \partial_{E''} \bar \partial \bar f = - \bar \partial_{E''} \partial \bar f
\\
\partial_{E'} \partial f = 0  & &   \bar \partial_{E'} \bar \partial f = 0 \nonumber \\
 \partial_{E''} \partial \bar f = 0  & &   \bar \partial_{E''} \bar \partial \bar f = 0.
\end{eqnarray}

For any map $f : M \rightarrow N $ between  K\"ahler manifolds, we have  
\begin{equation} \label{hessian}
\left|\partial_{E''}\bar \partial \bar f \right|^2=\left|\bar  \partial_{E''}\partial \bar f \right|^2=\left|\partial_{E'}\bar \partial f \right|^2.
\end{equation}
Indeed, 
the left equality follows from (\ref{easycommute}) and the right from the fact that conjugation is an isometry.



%

\subsection{Siu's curvature} \label{KtoK}
\begin{definition}
Let $N$  be a K\"ahler manifold and $R$ its complexified curvature tensor. We say $N$  has \emph{negative (resp.~non-positive)  complex sectional curvature}, if
\[
R(V, \bar W, W, \bar V) <  0\, (\mbox{resp.}\leq 0) \  \ \forall   V, W \in TN^{\CC}.
\]

In \cite{siu}, Siu introduced the following notion of negative curvature.
Recall that 
for local holomorphic coordinates $(w^i)$   of a  K\"ahler manifold $N$, the curvature tensor is  of type (1,1) and is given explicitly by
\[
R_{i\bar j k \bar l} =-\frac{\partial^2 h_{i \bar j}}{\partial w^k \partial \bar w^l}+h^{p\bar q} \frac{\partial h_{k\bar q}}{\partial w^i} \frac{\partial h_{p \bar l}}{\partial \bar w^j}
\]    
where $h$ is the K\"ahler metric on $N$.
We say $N$ has \emph{strongly negative (resp.~strongly semi-negative) curvature} if 
\[R_{i \bar j k \bar l} (A^i \overline{B^j} - C^i \overline{D^j}) \overline{ (A^l \overline{B^k} - C^l \overline{D^k})}<0 \ \mbox{(resp.~$\leq 0$)}.
\]
for arbitrary complex numbers $A^i, B^i, C^i, D^i$ when $A^i \overline{B^j} - C^i \overline{D^j} \neq 0$ for at least one pair of indices $(i,j)$.  \end{definition}

 \begin{remark}
A K\"ahler manifold $N$ is strongly semi-negative  if and only if it has non-positive complex
sectional curvature (cf.~\cite[Theorem 4.4]{liu-sun-yau}).
 \end{remark}

%
%
%

\begin{lemma}\label{branch}
Let $N$ be a K\"ahler manifold with K\"ahler form $\omega$ and of strongly semi-negative curvature.  Let $M$ be another K\"ahler manifold and  $f : M \rightarrow N $ be a smooth map.
If  $Q:M \rightarrow \RR$ is defined by setting
\begin{eqnarray*}
Q \frac{ \omega^n}{n!} =-R_{i \bar j k \bar l}  \bar \partial f^i \wedge \partial \bar{f}^j \wedge  \partial f^k \wedge \bar \partial \bar f^l \wedge \frac{\omega^{n-2}}{(n-2)!},
\end{eqnarray*}
then $Q \geq 0$.
\end{lemma}

\begin{proof}
At a point with normal coordinates in the domain
\begin{eqnarray*}
\lefteqn{
R_{i \bar j k \bar l}  \bar \partial f^i \wedge \partial \overline{f^j} \wedge  \partial f^k \wedge \bar \partial \,\overline{f^l} \wedge \frac{\omega^{n-2}}{(n-2)!}}\\
& = & \sum_{\alpha, \beta}(\sqrt{-1})^{n-2} R_{i \bar j k \bar l} \left(
-\partial_{\bar \alpha} f^i \partial_\alpha \overline{f^j} \partial_\beta f^k \partial_{\bar \beta} \overline{f^l}
	+ \partial_{\bar \alpha} f^i \partial_\beta \overline{f^j} 		\partial_\alpha f^k \partial_{\bar \beta} \overline{f^l} 
	\right.
\\
& & \ \left. +\partial_{\bar \beta} f^i \partial_\alpha \overline{f^j} \partial_\beta f^k \partial_{\bar \alpha} \overline{f^l}
	- \partial_{\bar \beta} f^i \partial_\beta \overline{f^j} 		\partial_\alpha f^k \partial_{\bar \alpha} \overline{f^l} \right) 
	\wedge \left( \wedge_\gamma (dz^\gamma \wedge d\bar z^\gamma) \right)
\\
&=&4 \sum_{\alpha, \beta} R_{i\bar j k \bar l}\left((\partial_{\bar \alpha} f^i)(\overline{\partial_\beta f^l})-(\partial_{\bar \beta}  f^i)(\overline{\partial  _\alpha f^l})\right)
\overline{\left((\partial_{\bar \alpha} f^j)(\overline{\partial_\beta f^k})-(\partial_{\bar \beta}  f^j)(\overline{\partial  _\alpha f^k})\right)} \frac{\omega^n}{n!}
\\
& = & 4 \sum_{\alpha, \beta} R_{i\bar j k \bar l}\left((\partial_{\bar \alpha} f^i)(\overline{\partial_\beta f^j})-(\partial_{\bar \beta}  f^i)(\overline{\partial  _\alpha f^j})\right)
\overline{\left((\partial_{\bar \alpha} f^l)(\overline{\partial_\beta f^k})-(\partial_{\bar \beta}  f^l)(\overline{\partial  _\alpha f^k})\right)} \frac{\omega^n}{n!}
\\
& \leq & 0.
\end{eqnarray*}
The last equality is because  $R_{\i \bar j k\bar l}=R_{i \bar l k \bar l}$, and the last inequality is because 
of the assumption that $N$ has strong semi-negative curvature.
\end{proof}

\subsection{Siu's Bochner Formula}

 \begin{theorem}[Siu-Bochner formula, \cite{siu} Proposition 2]
\label{siubochner}
For a harmonic map $f: M \rightarrow N$ between K\"{a}hler manifolds, 
\begin{eqnarray*}
 \partial  \bar \partial\{\bar \partial f,\bar \partial f \}\wedge \frac{ \omega^{n-2}}{(n-2)!} \nonumber
& = &    \left(4\left|\partial_{E'}\bar \partial f \right|^2 +Q\right) \frac{ \omega^n}{n!} 
. 
\end{eqnarray*}
\end{theorem}

\begin{proof}
Combine Lemma~\ref{branch}, Lemma~\ref{brac}, Lemma~\ref{aux} and Corollary~\ref{tree'} below.
\end{proof}

The curvature operators of $E'$ and $E''$ are $R_{E'}=-(\partial_{E'} + \bar \partial_{E'} )^2$ and $R_{E''}=-(\partial_{E''} + \bar \partial_{E''} )^2$ respectively.

\begin{lemma}\label{brac}For any smooth map $f : M \rightarrow N $ between  K\"ahler manifolds, we have
\[
\partial  \bar \partial\{\bar \partial f,\bar \partial f \}\wedge \frac{ \omega^{n-2}}{(n-2)!}=\left(-\{\partial_{E'}\bar \partial f,  \partial_{E'}\bar \partial f \}- \{\bar \partial f, R_{E'}( \bar \partial f) \}\right)\wedge \frac{ \omega^{n-2}}{(n-2)!}
\]
and
\[
\partial  \bar \partial\{\bar \partial \bar{f},\bar \partial \bar{f} \}\wedge \frac{ \omega^{n-2}}{(n-2)!}
= \left(-\{\partial_{E''}\bar \partial \bar{f},  \partial_{E''}\bar \partial \bar{f} \}- \{\bar \partial \bar{f}, R_{E''}( \bar \partial \bar{f}) \}\right)\wedge \frac{ \omega^{n-2}}{(n-2)!}.
\] 
\end{lemma}
\begin{proof}
By setting $\psi=\xi= \bar \partial f \in \Omega^{0,1}(E')$ in (\ref{bracket}) and repeatedly using the fact that $\bar \partial_{E'} \bar \partial f=0$ (cf.~(\ref{easycommute})), 
\begin{eqnarray*}
\partial  \bar \partial\{\bar \partial f,\bar \partial f \}
& =&-\{\partial_{E'}\bar \partial f,  \partial_{E'}\bar \partial f \}+ \{\bar \partial f, \bar \partial_{E'} \partial_{E'}\bar \partial f \}
\\
& =&-\{\partial_{E'}\bar \partial f,  \partial_{E'}\bar \partial f \}+ \{\bar \partial f, (\bar \partial_{E'} \partial_{E'} +   \partial_{E'} \bar \partial_{E'}+ \bar \partial_{E'}^2) \bar \partial f  \}. 
\end{eqnarray*}
Since $\{\bar \partial f, \partial_{E'}^2 \bar \partial f\} \wedge \frac{ \omega^{n-2}}{(n-2)!}$ is an 
$(n-1,n+1)$-form and hence zero  for dimensional reasons, we can complete the square to obtain
\begin{eqnarray}
\partial  \bar \partial\{\bar \partial f,\bar \partial f \}\wedge \frac{ \omega^{n-2}}{(n-2)!}
& =&\left(-\{\partial_{E'}\bar \partial f,  \partial_{E'}\bar \partial f \}+ \{\bar \partial f, (\partial_{E'} + \bar \partial_{E'} )^2 \bar \partial f\}\right) \wedge \frac{ \omega^{n-2}}{(n-2)!} \nonumber\\
& =&\left(-\{\partial_{E'}\bar \partial f,  \partial_{E'}\bar \partial f \}- \{\bar \partial f, R_{E'}( \bar \partial f) \} \right)\wedge \frac{ \omega^{n-2}}{(n-2)!} \nonumber
\end{eqnarray}
which proves the first equation.
The second equation follows by setting $\psi=\xi= \bar \partial \bar{f} \in \Omega^{0,1}(E'')$ in (\ref{bracket}) and following exactly the same computation.
\end{proof}

\begin{lemma}\label{aux}
For any harmonic map $f : M \rightarrow N $ between  K\"ahler manifolds, we have
\begin{eqnarray*}
-\{\partial_{E'}\bar \partial f,  \partial_{E'}\bar \partial f \} \wedge \frac{ \omega^{n-2}}{(n-2)!} & = &  
4\left|\partial_{E'}\bar \partial f \right|^2 \frac{ \omega^n}{n!} \\
-\{\partial_{E''}\bar \partial \bar{f},  \partial_{E''}\bar \partial \bar{f} \} \wedge \frac{ \omega^{n-2}}{(n-2)!} & = & 
4\left|\partial_{E''}\bar \partial \bar f \right|^2 \frac{ \omega^n}{n!}.
\end{eqnarray*}
\end{lemma}

\begin{proof}
Apply Lemma~\ref{stc} with $\phi=\partial_{E'}\bar \partial f$ (resp.~$\phi=\partial_{E''}\bar \partial \bar f$).   Since $f$ is harmonic, $\mbox{Tr}_{\omega} \partial_{E'}\bar \partial f=0$ and $\mbox{Tr}_{\omega} \partial_{E''}\bar \partial \bar f=0$ by  Lemma~\ref{trace}.
\end{proof}

\begin{lemma}\label{tree}
For any smooth map $f : M \rightarrow N $ between  K\"ahler manifolds, we have
\[
\{\bar \partial  f, R_{E'}  (\bar \partial  f) \}  = R_{i \bar j k \bar l} \bar \partial f^i \wedge \partial \bar{f}^j \wedge  \partial f^k \wedge \bar \partial \bar{f}^l=\{\bar \partial \bar f, R_{E''}  (\bar \partial \bar f) \}.
\]
\end{lemma}
\begin{proof}
Using normal coordinates, we compute
 \begin{eqnarray*}\label{dminor}
\{\bar \partial  f, R_{E'}  (\bar \partial  f) \}  
&=&
\{\bar \partial f^i\frac{\partial}{\partial f^i}, R_{E'}(\bar \partial f^j\frac{\partial}{\partial f^j})\} \nonumber \\
&=&
\{\bar \partial f^i\frac{\partial}{\partial f^i}, \bar \partial f^j \wedge R_{E'}(\frac{\partial}{\partial f^j})\} \nonumber \\
&=&
\{\bar \partial f^i\frac{\partial}{\partial f^i}, \bar \partial f^j \wedge R^s_{  j k \bar l}\partial f^k \wedge \overline{\partial f^l} \frac{\partial}{\partial f^s}\} \nonumber\\
 &=&R_{i \bar j \bar k l} \bar \partial f^i \wedge \partial \bar{f}^j \wedge  \bar \partial \bar f^k \wedge \partial f^l \\
 &=& R_{i \bar j k \bar l}  \bar \partial f^i \wedge \partial \bar{f}^j \wedge  \partial f^k \wedge \bar \partial \bar f^l 
  \nonumber
 \end{eqnarray*}
 which proves the first equality.  The second equality is proved similarly:
\begin{eqnarray*}\label{emajor}
\{\bar \partial  \bar{f}, R_{E''}  (\bar \partial  \bar{f}) \}  &=&
\{\bar \partial \bar{f}^i\frac{\partial}{\partial \bar{f}^i}, R_{E''}(\bar \partial \bar{f}^j\frac{\partial}{\partial \bar{f}^j})\} \nonumber \\ &=&
\{\bar \partial \bar{f}^i\frac{\partial}{\partial \bar{f}^i}, \bar \partial \bar{f}^j \wedge R_{E''}(\frac{\partial}{\partial \bar{f}^j})\} \nonumber \\
&=&\{\bar \partial \bar{f}^i\frac{\partial}{\partial \bar{f}^i}, \bar \partial \bar{f}^j \wedge R^{\bar s}_{\bar j \bar k l}\partial \bar{f}^k \wedge \bar \partial f^l  \frac{\partial}{\partial \bar{f}^s} \} \\
 &=&R_{\bar i  j k \bar l}  \bar \partial \bar{f}^i \wedge \partial f^j \wedge \bar \partial f^k \wedge \partial \bar{f}^l  \nonumber \\ 
  &=&R_{j  \bar i k \bar l}   \partial f^j \wedge \bar \partial \bar{f}^i \wedge \bar \partial f^k \wedge \partial \bar{f}^l  \nonumber \\
    &=&R_{i  \bar j k \bar l}   \partial f^i \wedge \bar \partial \bar{f}^j \wedge \bar \partial f^k \wedge \partial \bar{f}^l  \nonumber \\ 
  &=&R_{i \bar j k \bar l}  \bar \partial f^i \wedge \partial \bar{f}^j \wedge  \partial f^k \wedge \bar \partial \bar f^l. 
  \nonumber
 \ \ \ \ \ \ \ \  \qedhere
\end{eqnarray*}
 \end{proof}
 \begin{corollary}\label{tree'}
For any smooth map $f : M \rightarrow N $ between  K\"ahler manifolds, we have
\[
-\{\bar \partial  f, R_{E'}  (\bar \partial  f) \}  \wedge \frac{\omega^{n-2}}{(n-2)!}=Q \frac{\omega^n}{n!}.
\]
\end{corollary}

\begin{proof}
Combine Lemma~\ref{tree} with the definition of $Q$ given in Lemma~\ref{branch}.
\end{proof}

%

\begin{theorem} \label{siumainthm}
Suppose $M$ and $N$ are compact K\"ahler manifolds and the curvature of $N$ is strongly semi-negative. If
$f: M \rightarrow N$ is a harmonic map, then $f$ is pluriharmonic.  If, in addition, the curvature of $N$ is strongly negative and  the rank$_\RR df \geq 3$ at some point of $M$, then f is either  holomorphic or conjugate holomorphic.
\end{theorem}

\begin{proof}
Integrate Siu's Bochner formula over $M$. Applying Stoke's theorem results in the left hand side being 0.  The two terms on the right hand side are non-negative pointwise,  hence they must be identically equal to 0.  In particular, $\partial_{E'} \bar \partial f=0$; i.e.~$f$ is is pluriharmonic.  If the rank is $\geq 3$ at some point $x$, $\bar \partial f=0$ in some neighborhood of $x$ by the definition of $Q$. Hence $\bar \partial f=0$ in all of $M$.
\end{proof}

\subsection{Variations of the Siu and Sampson Formulas}

The following is a variation of the Sampson's Bochner Formula. For harmonic metrics, this is due to  Mochizuki (cf. \cite[Proposition 21.42]{mochizuki-memoirs}).  

 \begin{theorem} \label{mochizukibochnerformula'}
For a harmonic map $f: M \rightarrow N$ from a K\"{a}hler manifold to a Riemannian manifold, 
\begin{eqnarray*}
  d \{d'_E  d'f,  d''f -  d'f\} \wedge \frac{ \omega^{n-2}}{(n-2)!} 
 &= &  8\left(\left|d'_E d''f\right|^2+Q_0   \right)  \wedge \frac{\omega^n}{n!}.
  \end{eqnarray*}
\end{theorem}
\begin{proof} 
The key observation is that, since $d' \{ d'_E  d''f, d''f\} \wedge\frac{ \omega^{n-2}}{(n-2)!}$ is an $(n+1,n-1)$-form and $d'' \{d'_E d''f, d'f\}\wedge \frac{ \omega^{n-2}}{(n-2)!}$ is an $(n-1,n+1)$-form, these two forms are both identically equal  to zero.  Thus,
 \begin{eqnarray} 
d' \{  d'_E  d''f, d'f -d''f\} \wedge\frac{ \omega^{n-2}}{(n-2)!}
& = &  d' \{d'_E d''f,  d'f\} \wedge\frac{ \omega^{n-2}}{(n-2)!}
\nonumber \\
& = &-  d' \{d''_E d'f,  d'f \} \wedge\frac{ \omega^{n-2}}{(n-2)!}\ \ \ \mbox{(by (\ref{easycommute'}))}.\nonumber \\
  & = &
-  d' d'' \{ d'f, d'f \} \wedge \frac{ \omega^{n-2}}{(n-2)!}
  \nonumber
  \nonumber \\
  & = & d' d''\{d''f, d''f\} \wedge \frac{ \omega^{n-2}}{(n-2)!}
\label{sb1'}
 \\
 d'' \{d'_Ed''f, d'f - d''f\} \wedge \frac{ \omega^{n-2}}{(n-2)!} & = & - d'' \{d'_Ed''f, d''f\} \wedge \frac{ \omega^{n-2}}{(n-2)!} 
  \nonumber \\  & = &  -  d'' d'\{d''f, d''f\} \wedge \frac{ \omega^{n-2}}{(n-2)!} \nonumber \\
& = & d' d'' \{d''f, d''f\} \wedge \frac{ \omega^{n-2}}{(n-2)!}.
 \label{sb2'}
 \end{eqnarray}
Thus,
\begin{eqnarray*} \label{mochizukitrick}
d \{d''_E  d'f,  d''f -  d'f\} \wedge \frac{ \omega^{n-2}}{(n-2)!}   & = &  d \{d'_E d''f, d'f - d''f\} \wedge \frac{ \omega^{n-2}}{(n-2)!}
 \ \ \ \mbox{(by (\ref{easycommute'}))}
 \\
 & = &  (d'+d'') \{d'_E d''f, d'f - d''f\} \wedge \frac{ \omega^{n-2}}{(n-2)!}
 \\
& = & 2d' d''\{d''f, d''f\} \wedge \frac{ \omega^{n-2}}{(n-2)!}  \ \ \mbox{(by (\ref{sb1'}) and (\ref{sb2'}))}.
 \end{eqnarray*} 
 Thus, the asserted identity follows from
Theorem~\ref{sampsonbochner}. 
\end{proof}

By applying a similar proof as Theorem~\ref{mochizukibochnerformula}, we obtain a variation of the Siu's Bochner formula.
\begin{theorem} \label{mochizukibochnerformula}
For a harmonic map $f: M \rightarrow X$ between K\"{a}hler manifolds, 
\begin{eqnarray*}
  d \{\bar \partial_{E'}  \partial f,  \bar \partial f -  \partial f\} \wedge \frac{ \omega^{n-2}}{(n-2)!} 
 &= &  \left(8\left|{\partial}_{E'} \bar \partial f\right|^2+2Q   \right)  \wedge \frac{\omega^n}{n!}.
  \end{eqnarray*}
\end{theorem}

\begin{proof}
As in the proof of Theorem~\ref{mochizukibochnerformula'},   $\partial \{ \partial_{E'}  \bar\partial f, \bar \partial f\} \wedge\frac{ \omega^{n-2}}{(n-2)!}
=0=\bar \partial \{\partial_{E'} \bar \partial f, \partial f\}\wedge \frac{ \omega^{n-2}}{(n-2)!}$ and hence
\begin{eqnarray*} 
\partial \{  \partial_{E'}  \bar \partial f, \partial f -\bar  \partial f\} \wedge\frac{ \omega^{n-2}}{(n-2)!}
  & = & \partial \bar \partial\{\bar \partial \bar f, \bar \partial \bar f\} \wedge \frac{ \omega^{n-2}}{(n-2)!}
  \label{sb1},
 \\
 \bar\partial \{\partial_{E'}\bar \partial f, \partial f - \bar \partial f\} \wedge \frac{ \omega^{n-2}}{(n-2)!} 
& = & \partial \bar\partial\{\bar \partial f, \bar \partial f\} \wedge \frac{ \omega^{n-2}}{(n-2)!}.
 \label{sb2}
 \end{eqnarray*}
Consequently,
\begin{eqnarray*} \label{mochizukitrick}
d \{\bar \partial_{E'}  \partial f,  \bar \partial f -  \partial f\} \wedge \frac{ \omega^{n-2}}{(n-2)!} 
  & = &  d \{\partial_{E'} \bar\partial f, \partial f - \bar \partial f\} \wedge \frac{ \omega^{n-2}}{(n-2)!}
 \\
 & = &  (\partial+\bar \partial) \{\partial_{E'} \bar\partial f, \partial f - \bar \partial f\} \wedge \frac{ \omega^{n-2}}{(n-2)!}
 \\
& = &  \left( \partial \bar\partial\{\bar \partial f, \bar \partial f\}+ \partial \bar \partial\{\bar \partial \bar f, \bar \partial \bar f\}  \right)\wedge \frac{ \omega^{n-2}}{(n-2)!}.
 \end{eqnarray*} 
The asserted identity follows from
Theorem~\ref{siubochner}. 
\end{proof}

\subsection{Pluriharmonic maps into Euclidean buildings}

\begin{theorem} \label{thm:GS}
    		Let $M$ be  a compact K\"ahler manifold and $\Delta(G)$ be the Bruhat-Tits building associated to a  semisimple algebraic group $G$ defined over a  non-Archimedean local field $K$.   For any Zariski dense representation of $\rho:\pi_1(X)\to G(K)$, there exists a $\rho$-equivariant, locally Lipschitz  pluriharmonic map $f:\tilde M \rightarrow \Delta(G)$ from the universal cover $\tilde{M}$. 
\end{theorem}

\begin{definition}	A {\it Euclidean   building} of dimension $n$ is a piecewise Euclidean  simplicial complex $\Delta$ such that:
	\begin{itemize}  
		\item  $\Delta$ is the union of a collection $\mathcal{A}$ of subcomplexes $A$, called apartments, such that the intrinsic metric $d_{A}$ on $A$ makes $\left(A, d_{A}\right)$ isometric to the Euclidean space $\mathbb{R}^{n}$ and induces the given Euclidean metric on each simplex.
		\item  Given two apartments $A$ and $A^{\prime}$ containing both simplices $S$ and $S^{\prime}$, there is a simplicial isometry from $\left(A, d_{A}\right)$ to $\left(A^{\prime}, d_{A^{\prime}}\right)$ which leaves both $S$ and $S^{\prime}$ pointwise fixed.\item $\Delta$ is locally finite.
	\end{itemize} 
\end{definition}

\begin{definition}
A point $x_0$ is said to be a {\it regular point} of a harmonic map $f$, if there exists $r>0$ such that $f(B_r(x_0))$ of $x$ is contained in an apartment of $\Delta$.  A {\it singular point} of $f$ is a point of $\Omega$ that is not a regular point.   
The regular (resp.~singular) set $\mathcal R(f)$ (resp.~$\mathcal S(u)$)  of $f$ is the set of all regular (resp.~singular) points of $f$.
\end{definition}

\begin{example}
Consider a measured foliation defined by the quadratic differential 
$zdz^2$ on $\CC$.  The leaves of the horizontal foliation define a  3-pod $T$ and the  transverse measure gives $T$ a distance function $d$ making $(T,d)$ into a NPC space.  The projection along the vertical foliation $u:\CC \rightarrow T$ is a harmonic map. The leaf containing $0$ is a non-manifold point of $T$.  Let  $K=u^{-1}(0)$. Then $K$ is also a 3-pod.  On the other hand, every point of $K$ besides $0$ has a neighborhood mapping into an isometric copy of $\RR$ and $\mathcal S(0)=\{0\}$.  In particular, the singular set is of Hausdorff codimension 2. Similarly one can construct harmonic maps to other homogeneous trees  by taking quadratic differentials of higher order.
\end{example}

The next two theorems are proved in \cite{gromov-schoen}.

\begin{theorem} \label{regbuildings}
The singular set  ${\mathcal S}(f)$ of a harmonic map $f:M \rightarrow \Delta$ is a closed set of Hausdorff codimension $\geq 2$.
\end{theorem}

\begin{theorem}\label{goto0}
Let   $f:M \rightarrow \Delta$ be as in Theorem~\ref{regbuildings}.  There exists a sequence of smooth functions $\psi_i$  with $\psi_i \equiv 0$ in a neighborhood of ${\mathcal S}(u)$, $0 \leq \psi_i \leq 1$ and $\psi_i(x) \rightarrow 1$ for all $x \in {\mathcal S}(u)$ such that
\[
\lim_{i \rightarrow \infty} \int_M |\nabla \nabla u| |\nabla \psi_i| \ d\mu =0.
\]
\end{theorem}

%

By Theorem~\ref{regbuildings}, Siu's or Sampson's Bochner formula holds at a.e.~$x \in \tilde M$.   We now follow the proof of Theorem~\ref{siumainthm} where  integration by parts can be  justified using Theorem~\ref{regbuildings} and Theorem~\ref{goto0}.

\newpage 
\section{Donaldson Corlette theorem}
\label{sec:Donaldson}

\subsection{Introduction:  Higgs bundles via harmonic maps}
In this lecture, we prove the theorem of Donaldson and Corlette relating  harmonic maps to symmetric spaces of non-compact type and flat connections. We do it explicitly for $\mathsf{SL}(n,\CC)$. This correspondence is very well known and there are many excellent references to consult.  Given the interest of the audience in this subject, we decided to give all the details of the proof explicitly.
 See also \cite{donaldson}, \cite{corlette2} and the expositional paper \cite{li}.  
 
 \subsection{The flat vector bundle associated to a representation}
Let $\rho:\pi_1(M) \rightarrow G=\mathsf{SL}(n,\CC)$ be a homomorphism and 
$$
E=\tilde M \times_\rho \CC^n \rightarrow M
$$ be the associated flat vector bundle.
 Let  
 $\mathcal H $
  denote the space of positive definite self-adjoint matrices of determinant one.   For $g \in \mathsf{SL}(n,\CC)$, define an action $\mathcal A_g$ on the space of $(n\times n)$-matrices $M_{n\times n}(\CC)$ by
 \begin{equation} \label{Ag}
\mathcal A_g(h)=g^{-1*}hg^{-1}.
 \end{equation}
Note that  $\mathcal H$ is invariant  under $\mathcal A_g$ and hence it defines an action on $\mathcal H$.

   A  $\rho$-equivariant map $h:\tilde M \rightarrow \mathcal H$ defines a hermitian metric  on $E$ by first defining 
\begin{equation}\label{hermnmet}
H(s,t)={\bar s}^tht
\end{equation}
on the universal cover  $\tilde M \times \CC^n$ and 
 descending to a metric on $E$ by equivariance. 

Given the flat vector bundle $(E,d)$ defined by $\rho$ and the Hermitian metric $H$ defined by a $\rho$-equivariant map, we define $\theta \in \Omega^1(M, \mathsf{End}(E))$ by the formula
\begin{equation}\label{higgs}
H(\theta s,t)=1/2 \left(H(ds,t)+H(s,dt)-dH(s,t)\right)
\end{equation}
and $D$ by the formula
\begin{equation}\label{unicon}
d=D+\theta.
\end{equation}
Formulas~(\ref{higgs}) and~(\ref{unicon}) immediately imply that
\begin{eqnarray}\label{sfadj}
H(\theta s,t)=H( s, \theta t)
\end{eqnarray} 
and 
\begin{eqnarray}\label{compa}
H(Ds,t)+H(s,Dt)&=&
H(ds,t)-H(\theta s,t)+H(s,dt)-H( s, \theta t) \nonumber\\
&=&dH(s,t).
\end{eqnarray}
In other words, $D$ is a Hermitian connection on $(E,H)$. 

We claim
\begin{eqnarray}\label{compa}
\theta=-\frac{1}{2}h^{-1}dh.
\end{eqnarray}
To see (\ref{compa}), compute
\begin{eqnarray*}\label{Fgtysm}
dH(s,t)&=&{ d\bar s}^th \, t+{\bar s}^tdh \, t+{\bar s}^th \, dt\\
&=&H( d s, t)+{\bar s}^tdh \, t+H( s, dt)\\
&=&H( D s, t)+H(\theta s,t)+{\bar s}^tdh \, t +H(s, D t)+H(s,\theta t)\\
&=&dH(s,t)+ H( \theta s, t)+ H(s, h^{-1}dh \,t) + H(s, \theta t).
\end{eqnarray*}
Thus,
\begin{eqnarray*}
H( \theta s, t)=H( s, \theta t)=-1/2H(s, h^{-1}dh \, t)
\end{eqnarray*}
and (\ref{compa}) follows.

Let $\mathsf{End}_0(E)$ denote the space of trace-less endomorphisms of $E$.   We claim that $D$ is a 
$\mathsf{SL}(n,\CC)$-connection and 
$\theta \in \mathsf{End}_0(E).$ By (\ref{unicon}) and since $d$ is traceless, it suffices to show that $\theta$ is traceless. Indeed, since $G/K$ is a Cartan-Hadamard space, we can write $h=e^u$ over a simply connected region $U$ in $M$ (or passing to the universal cover) where $u(x) \in \mathfrak p$ for all $x \in U$. Thus,
\[
\theta=h^{-1}dh=du
\]
is traceless since $u$ is traceless.

As connections on $\mathsf{End}_0(E)$,
\begin{equation} \label{DonEnd}
D=d+\frac{1}{2}\left[h^{-1}dh, \  \cdot \ \ \right].
\end{equation}
We  apply harmonic map theory to prove:  


\begin{theorem} \label{thm:donaldson}
Given  an irreducible representation  $\rho:\pi_1(M) \rightarrow \mathsf{SL}(n,\CC)$, there exists a $\rho$-equivariant map $h:\tilde M \rightarrow \mathcal H$ such that for the  Hermitian metric $H$, Hermitian connection $D$ on $\mathsf{End}_0(E)$ and $\theta \in \Omega^1(M, \mathsf{End}_0(E))$ defined by (\ref{hermnmet}), (\ref{higgs}) and (\ref{unicon}) respectively, 
\begin{equation}\label{harm*}
d_D^{\star} \theta=0.
\end{equation}
\end{theorem}

The proof of Theorem~\ref{thm:donaldson} is given several steps:  (1) Choose $h$ to be a harmonic map (cf.~Section~\ref{sec:ehm}). (2) Show that the Hermitian connection $D$ is related to the Levi-Civita connection on $\mathcal H$  (cf.~Section~\ref{hlc}). (3)
Show that the harmonic map equation for $h$ is equivalent to (\ref{harm*})  (cf.~Section~\ref{completion31}).

\subsection{The equivariant map $h$ is harmonic}
\label{sec:ehm}

The first step in the proof of Theorem~\ref{thm:donaldson} is to choose the map $h$ of Theorem~\ref{thm:donaldson}  as a harmonic map into $(\mathcal H, g_\mathcal H)$ where the metric is given by
\[
g_{\mathcal H}(X, Y)=\frac{n}{2}\mathsf{trace}(h^{-1}X h^{-1}Y) \ \mbox{for} \ X, Y \in T_h\mathcal H.
\]
\begin{definition}
We call $h$ or $H$ a \emph{harmonic metric}.
\end{definition}
For $G=\mathsf{SL}(n,\CC)$, $K=\mathsf{SU}(n)$, let $\mathfrak{sl}(n) =\mathfrak k \oplus \mathfrak p$ be the Cartan decomposition
and
\[
B(X,Y)=2n\, \mathsf{trace}(XY)
\]
be the Killing form on $\mathfrak{sl}(n)$. The inner product is positive definite on $\mathfrak p$.

Let $L_g:G/K \rightarrow G/K$ be  left multiplication and define a metric $g_{G/K}$ on $G/K$ by  metrically identifying
\begin{equation} \label{metricallyid}
 dL_{g^{-1}}: T_{gK}G/K \rightarrow T_{eK}G/K=\mathfrak p.
\end{equation}
This defines $(G/K,g_{G/K})$ as a symmetric space of non-compact type. 

\begin{lemma}\label{2models}
The map 
\begin{eqnarray*}
\Psi: G/K & \mapsto &  \mathcal H
\\
gK & \mapsto & g^{-1*}g^{-1}=h 
\end{eqnarray*}
identifies $(G/K,g_{G/K})$ isometrically with $(\mathcal H, g_{\mathcal H})$ as $G$-spaces.  
\end{lemma}
\begin{proof}
First, $\Psi$ is equivariant with respect to the action $L_g$ on $G/K$ and the action $\mathcal A_g$ on $\mathcal H$.  Indeed,
\begin{align*}
\Psi \circ L_g (g_1K) & = \Psi(gg_1K) 
= (gg_1)^{-1*}(gg_1)^{-1} 
 = g^{-1*}(g_1^{-1*}g_1^{-1})g^{-1} 
 \\
 & = g^{-1*}\Psi(g_1K)g^{-1}=\mathcal A_g \circ \Psi (g_1K).
\end{align*}
Second,  $\Psi$ is an isometry. Since the metric $g_{G/K}$ is defined by (\ref{metricallyid}), 
we need to show that with $h=g^{-1*}g^{-1} \in \mfld$, 
$$
d(L_{g^{-1}})_{gK} \circ d(\Psi^{-1})_h  = d\big( (\Psi \circ L_g)^{-1} \big)_h: T_h \mathcal H \rightarrow T_{eK}G/K=\mathfrak p
$$
is an isometry.  This is a straightforward calculation:  
Let $t \mapsto g_t$ be a path in $G/K$ with $g_0=eK$ and $\dot g_0 \in T_{eK}G/K$ (where dot indicates the $t$-derivative). For $\dot g \in T_{eK} G/K$, since  $\dot g_0$ is self-adjoint,
\begin{eqnarray}\label{helpder}
(d\Psi)_e (\dot g_0)=\frac{d}{dt}\Big|_{t=0} (g_t^{-1*}g_t^{-1})=-\dot g_0^*-\dot g_0=-2\dot g_0.
\end{eqnarray}
For $X \in T_h\mathcal H$,
\begin{eqnarray*}
d\big( (\Psi \circ L_g)^{-1} \big)_h (X) & = & d \big( (\mathcal{A}_g \circ \Psi)^{-1} \big)_h (X)
 =  d(\Psi^{-1} \circ \mathcal{A}_{g^{-1}})_h(X)\\
& = & d(\Psi^{-1} \circ \mathcal{A}_{g^{-1}})_{g^{-1*}g^{-1}}(X)
 =  (d\Psi_e)^{-1} \circ (d\mathcal{A}_{g^{-1}})_{g^{-1*}g^{-1}}(X)\\
& = & (d\Psi^{-1})_e (g^*Xg)=-\frac{1}{2}g^*Xg
\\
& = & -\frac{1}{2}Ad_{g^{-1}}(gg^*X)=-\frac{1}{2}Ad_{g^{-1}}(h^{-1}X).
\end{eqnarray*}
Here we used (\ref{helpder}) in the third to last equality.
Using this formula and the $\mathsf{Ad}$-invariance of the Killing form, we have for   $X, Y \in T_h\mfld$ 
\begin{align*}
B\left( d\big( (\Psi \circ L_g)^{-1} \big)_h (X), d\big( (\Psi \circ L_g)^{-1} \big)_h (Y) \right)& = \frac{1}{4} B(h^{-1}X, h^{-1}Y)
\\
& =  \frac{n}{2} \mathsf{trace}(h^{-1}Xh^{-1}Y)
\\
& =  g_{\mfld}(X,Y). \qedhere
\end{align*}
\end{proof}

By Theorem~\ref{existence}, there  exists a $\rho$-equivariant  harmonic map $f:  \tilde M \rightarrow G/K$.  In view of the Lemma~\ref{2models}, we identify $G/K$ with $\mathcal H$ and obtain   
 a $\rho$-equivariant  harmonic map  
 \[
 h=f^{-1*}f^{-1}: \tilde M \rightarrow \mathcal H, \ \ \ 
d_\nabla^{\star} dh =0
\]
where $\nabla$ is the pullback to $h^*T\mathcal H$ of the  Levi-Civita connection of $(\mathcal H, g_{\mathcal H})$.

\subsection{The hermitian connection $D$ and  the Levi-Civita connection on $(\mathcal H,g_\mathcal H)$}
\label{hlc}

Recall that $\mathcal H \subset \mathsf{SL}(n,\CC)$ is the space of positive definite, self-adjoint matrices of determinant one and consider the map
\[
P_h: T_h \mathcal H \rightarrow P_h(T_h \mathcal H)  \subset \mathfrak{sl}(n),  \ \ \ X \mapsto h^{-1}X
\]
whose  image $P_h(T_h \mathcal H)$  consists of matrices self-adjoint with respect to $h$.
Indeed,
\begin{eqnarray*}
(h^{-1}X)^{*_h}=h^{-1}(h^{-1}X)^*h=h^{-1}X.
\end{eqnarray*}
Extending this map complex linearly  induces an isomorphism 
\[
P_h^{\CC}: T_h \mathcal H^{\CC} \xrightarrow{\simeq}\mathfrak{sl}(n)
\]
 which defines a global isomorphism
\begin{equation} \label{pc}
P^{\CC}: T \mathcal H^{\CC} \xrightarrow{\simeq} \mathcal H \times \mathfrak{sl}(n).
\end{equation}
%
%

The trivial connection $d$ on $\mathcal H \times \mathfrak{sl}(n)$ pulls back by the isomorphism $P^{\CC}$ to a flat connection $\bar \nabla$ on $T\mathcal H^{\mathbb C}$; i.e.
\[
\bar \nabla_XY \Big|_h={P^{\CC}}^{-1}\circ d_X\circ P^{\CC} (Y)\Big|_h.
\]
We next  compute the formula for $\bar \nabla$ with respect to  the coordinates that  identify the space of  $(n\times n)$-matrices $M_{n\times n}(\CC)$  with $\RR^{n^2}$.  Let  $t \mapsto h_t$ be a curve in $\mathcal H$ with $h_0=h$ and $\dot h_0=X(h)$.  We have
\[
d_X(P^{\mathbb C}(Y))=\frac{d}{dt}\Big|_{t=0} \left(h^{-1}_t\, Y(h_t) \right) =h^{-1} \frac{d}{dt}\Big|_{t=0} Y(h_t) - h^{-1} \dot h_0 h^{-1} \, Y(h_0).
\]
Using the embedding $\mathcal H \hookrightarrow M_{n\times n}(\CC)$, we express $h_t=(h^{ij}_t)$.  Furthermore, we can express   $X=X^{ij} \partial_{ij}$ and $Y=Y^{kl} \partial_{kl}$ with respect to the coordinate basis $(\partial_{ij})$. 
Extending  $Y=Y^{kl} \partial_{kl}$ as a vector field on $M_{n\times n}(\CC)$, we apply the chain rule to obtain
\begin{equation} \label{dYdX}
\frac{d}{dt}\Big|_{t=0} Y(h_t)= \frac{d}{dt}\Big|_{t=0} Y^{kl}(h_t) \partial_{kl} =\left( \partial_{ij} Y^{kl}\Big|_h \dot h^{ij}_0 \right) \partial_{kl} = \left(X^{ij} \partial_{ij} Y^{kl} \right)\Big|_h \, \partial_{kl}= \frac{\partial Y}{\partial X}\Big|_h.
\end{equation}
Thus, the formula for the flat connection at the point $h \in \mathcal H$ is
\[
\bar \nabla_X Y = \frac{\partial Y}{\partial X} -Xh^{-1}Y.
\]

The Levi-Civita connection on $T \mathcal H$, denoted by $\nabla$ and  extended complex linearly to $T \mathcal H^{\CC}$, is given at $h \in \mathcal H$ by the formula
\[
\nabla_XY  =\frac{\partial Y}{\partial X} -\frac{1}{2}\left( X h^{-1}Y+ Yh^{-1}X\right).
\]
Indeed:
\\
(i)  $\nabla$ is torsion free:  First, for a function $f$ defined near $h$, 
\begin{eqnarray*}
\left( \frac{\partial Y}{\partial X} - \frac{\partial X}{\partial Y} \right) f & = & 
(X^{ij} \partial_{ij} Y^{kl}) \partial_{kl} f - (Y^{kl}\partial_{kl} X^{ij}) \partial_{ij} f
\\
& = & 
(X^{ij} \partial_{ij} Y^{kl}) \partial_{kl} f + X^{ij}Y^{kl} \partial_{ij} \partial_{kl} f - (Y^{kl} \partial_{kl} X^{ij}) \partial_{ij} f - Y^{kl} X^{ij} \partial_{kl} \partial_{ij} f
\\
& = & X(Y f)-Y(X f) =[X,Y]f.
\end{eqnarray*}
Thus, 
\begin{eqnarray*}
\nabla_X Y- \nabla_Y X & = & \left( \frac{\partial Y}{\partial X}-\frac{1}{2}(Xh^{-1}Y+ Yh^{-1}X) \right)- \left( \frac{\partial X}{\partial Y}-\frac{1}{2}(Yh^{-1}X+ Xh^{-1}Y) \right)
\\
& = & \frac{\partial Y}{\partial X} - \frac{\partial X}{\partial Y} =[X,Y].
\end{eqnarray*} 
(ii) $\nabla$ is metric compatible:  Using the path $t \mapsto h_t$  given above and using (\ref{dYdX}), 
\begin{eqnarray*}
\lefteqn{X g_{\mathcal H}(Y,Z) }
\\
& = & \frac{n}{2} \mathsf{trace}\left(\frac{\partial}{\partial t}\Big|_{t=0} \left(h^{-1}_tY(h_t) h^{-1}_t Z(h_t) \right) \right)
\\
& = &  \frac{n}{2} \mathsf{trace}\left[ \left( 
h^{-1} \frac{\partial Y}{\partial X} -h^{-1}Xh^{-1} Y  \right) h^{-1} Z+ h^{-1}Y \left( h^{-1} \frac{\partial Z}{\partial X} -h^{-1}X h^{-1}Z  \right)\right]
\\
& = &  \frac{n}{2} \mathsf{trace}\left[  
h^{-1} \left(\frac{\partial Y}{\partial X} -\frac{1}{2} \left( Xh^{-1} Y  +Yh^{-1}X \right) \right) h^{-1} Z \right]
\\
& & \ \ +  \frac{n}{2} \mathsf{trace}\left[ h^{-1}Y  h^{-1} \left(\frac{\partial Z}{\partial X} -\frac{1}{2} (X h^{-1}Z+Zh^{-1}X)  \right)\right]
\\
& = & g_{\mathcal H}(\nabla_X Y,Z)+ g_{\mathcal H}(Y, \nabla_X Z).
\end{eqnarray*}
The difference of the flat connection $\bar \nabla$ and the Levi-Civita connection $\nabla$  on $T\mathcal H^{\mathbb C}$ is
\begin{equation} \label{difference}
\left(\bar \nabla_XY- \nabla_XY \right) =\frac{1}{2}\left(Yh^{-1}X-Xh^{-1}Y \right)= -\frac{1}{2}h \left[h^{-1}X, h^{-1}Y\right].
\end{equation}

Let $\hat \nabla=P^{ \mathbb C} \circ \nabla \circ {P^{ \mathbb C}}^{-1}$  denote the pullback to $\mathcal H \times \mathfrak{sl}(n)$  of the Levi-Civita connection $\nabla$ on $T\mathcal H^\mathbb C$ via (\ref{pc}).  Then  the corresponding formula to (\ref{difference}) for the difference between the flat connection $d$ and  $ \hat \nabla$  on $\mathcal H \times \mathfrak{sl}(n) \rightarrow \mathcal H$ is 
\begin{equation} \label{conn}
d_X- \hat \nabla_X  =-\frac{1}{2}\left[h^{-1}X, \  \cdot \ \ \right].
\end{equation}
The bundle $\mathcal H \times \mathfrak{sl}(n)$ pulls back by  $h: \tilde M \rightarrow \mathcal H$ to the trivial $\mathsf{SL}(n,\CC)$-bundle $h^*(\mathcal H \times \mathfrak{sl}(n))$ on the universal cover $\tilde M$.  
\[ \begin{tikzcd}
h^*(\mathcal H \times \mathfrak{sl}(n)) \arrow{r}{} \arrow[swap]{d}{} & \mathcal H \times \mathfrak{sl}(n) \arrow{d}{} \\%
\tilde M \arrow{r}{h}&  \mathcal H
\end{tikzcd}
\]
From (\ref{conn}), the difference of the flat connection $d$ and the pullback $h^*\hat \nabla$    is given by the formula
\begin{eqnarray}\label{ias}
d_V-(h^*\hat \nabla)_V=-\frac{1}{2}\left[h^{-1}dh(V), \  \cdot \ \ \right].
\end{eqnarray}

Next, the pullback to $\tilde M$ of the endormorphism bundle $\mathsf{End}_0(E)$ is isomorphic to the  trivial bundle. 
 Taking the quotient by the induced action from $\rho$,  
$\mathsf{End}_0(E) \simeq   h^*(\mathcal H \times_\rho  \mathfrak{sl}(n)) \rightarrow M$
and the connection $h^*\hat \nabla$ induces a connection  on  $\mathsf{End}_0(E)$ (which we also call $\hat \nabla$).  

From (\ref{ias}), we have
\begin{equation*}\label{conn**}
\hat \nabla=d+\frac{1}{2}\left[h^{-1}dh, \  \cdot \ \ \right].
\end{equation*}
Hence, 
\begin{equation*}\label{conn**}
\hat \nabla=D
\end{equation*}
 by (\ref{DonEnd}).  In other words,   $D$ is the connection on $\mathsf{End}_0(E)$ induced by the Levi-Civita connection on $T^\mathbb C\mathcal H$.

\subsection{Completion of the proof of Theorem~\ref{thm:donaldson}}
\label{completion31}

The bundle isomorphism  ${P^\mathbb C}^{-1}$
 of (\ref{pc}) induces a bundle isomorphism (still denoted by ${P^\mathbb C}^{-1}$)
\begin{eqnarray*}
h^*(\mathcal H \times_\rho \mathfrak{sl}(n))  & \simeq &h^*(T\mathcal H^\mathbb C)  \rightarrow M
\\
\phi & \mapsto & h\phi. \nonumber
\end{eqnarray*}
Also,
\begin{eqnarray} \label{overM}
\hat \nabla={P^\mathbb C}\circ \nabla \circ {P^\mathbb C}^{-1}.
\end{eqnarray}
In particular, since
\[
\theta=-\frac{1}{2} h^{-1} dh \in \Omega^1(M, \mathsf{End}_0(E)) \simeq \Omega^1(M, h^*(\mathcal H \times_\rho \mathfrak{sl}(n))),
\]
we have
\begin{equation}\label{theta***}
{P^\mathbb C}^{-1}\theta=h\theta = -\frac{1}{2} dh \in \Omega^1(M, h^*(T\mathcal H^\mathbb C)).
\end{equation}
Theorem~\ref{thm:donaldson} follows from the  the following implications:
\begin{eqnarray*}
\mbox{$h$ is harmonic}  &  \Rightarrow &  0=-\frac{1}{2}d_\nabla^*dh
=d_\nabla^*h \theta=d_\nabla^*{P^\mathbb C}^{-1}\theta \ \ \mbox{ by  (\ref{theta***})}
\\
& \Rightarrow &
0 = {P^\mathbb C}d_\nabla^*{P^\mathbb C}^{-1}\theta=d_{\hat \nabla}^* \theta = d_D^* \theta  \ \ \mbox{ by  (\ref{conn**}) and (\ref{overM})}.
\end{eqnarray*}

\end{document}